%% file: pap.tex
\documentclass[final]{article}

\usepackage{amsmath}
\usepackage{amsthm}
\usepackage{amsfonts}
\usepackage{color}
\usepackage{subfig}
\usepackage{graphicx}
\usepackage[margin=1in,headsep=.2in]{geometry}
\usepackage{hyperref}
\hypersetup{
    colorlinks=true,
    linkcolor=blue,
    citecolor=blue,
    filecolor=magenta,
    urlcolor=cyan,
}

\usepackage{algorithmic,algorithm}
\usepackage{enumitem}
\usepackage{amssymb}
\usepackage{comment}
\input{commands}

\usepackage{mathrsfs}

\usepackage{tikz}
\usetikzlibrary{arrows.meta}
\usepackage{pgfplots}
\pgfplotsset{compat=1.13}

\title{Component-wise reduced order model lattice--type structure design}

\author{Sean McBane\\University of Texas, Austin\footnote{Ph.D. Candidate, Oden
Institute for Computational Engineering \& Sciences}\vspace{12pt}\\
Youngsoo Choi\\Lawrence Livermore National Laboratory\footnote{Lawrence
Livermore National Laboratory is operated by Lawrence Livermore National
Security, LLC, for the U.S. Depart- ment of Energy, National Nuclear Security
Administration under Contract DE-AC52-07NA27344.}
}
\date{}

\begin{document}

\maketitle

\begin{abstract}
  Lattice--type structures can provide a combination of stiffness with light
  weight that is desirable in a variety of applications. Design optimization of
  these structures must rely on approximations of the governing physics to
  render solution of a mathematical model feasible. In this paper, we propose a
  topology optimization (TO) formulation that approximates the governing physics
  using component-wise reduced order modeling as introduced in
  \cite{huynh2013static} and \cite{eftang2013pr}, which can reduce solution time
  by multiple orders of magnitude over a full-order finite element model while
  providing a relative error in the solution of $<$1\%. In addition, the offline
  training data set from such component-wise models is reusable, allowing its
  application to many design problems for only the cost of a single offline
  training phase, and the component-wise method is nearly embarrassingly
  parallel. We also show how the parameterization chosen in our optimization
  allows a simplification of the component-wise reduced order model (CWROM) not
  noted in previous literature, for further speedup of the optimization process.
  The sensitivity of the compliance with respect to the particular
  parameterization is derived solely in the component level.  In numerical
  examples, we demonstrate a 1000x speedup over a full-order FEM model with
  relative error of $<$1\% and show minimum compliance designs for two different
  cantilever beam examples, one smaller and one larger.  Finally, error bounds
  for displacement field, compliance, and compliance sensitivity of the CWROM
  are derived.  
\end{abstract}


\section{Introduction}\label{sec:intro}
Many systems in nature, such as bones, shells, and honeycombs, rely on intricate
lattice--type structure designs that are strong and lightweight.  Advances in
additive manufacturing have enabled industries to fabricate microstructures with
these qualities for a range of products. Lattice structures have also been used
to design materials with unusual properties, e.g., a material with negative
Poisson ratios. The most accurate way of modeling a lattice structure is to use
a finite element method (FEM) with a conforming mesh, which provides detailed
physics information.  However, this method introduces a large number of degrees
of freedom that may make the computational solution too expensive even with
access to high-performance computing facilities. In particular, design
optimization requires many simulations as it explores a parameter space, which
is even more formidable than running a single simulation. Therefore, most 3D CAD
software companies adopt approximation schemes, such as homogenization or beam
element-based approaches, in their lattice design tools. For example, Autodesk
Within is a popular commercial software for lattice structure design that uses
the beam/shell element-based method \cite{mahdavi2003evolutionary}. Their
website can be found in \cite{within}.  A start-up company, nTopology
\cite{ntopology}, introduces implicit geometric representation for lattice
design which is based on the beam elements.  Ansys \cite{ansys} and COMSOL
\cite{comsol} use the homogenization-based method.  Both Meshify \cite{meshify}
and 3DXpert \cite{3dxpert} use both beam/shell element and homogenization-based
methods. We present a method that improves on these approximation schemes
through the use of component-wise reduced order modeling to make solution of an
ordinary FEM model of a lattice efficient enough for the many-query context of
design optimization.

In homogenization-based methods \cite{bendsoe1988generating}, the material is
taken to be periodic, composed of unit cells whose properties are determined by
a high-fidelity model. Using the effective macroscale properties computed by
this model of the microscale, continuum topology optimization algorithms are
then used to develop a macro-level structure design. Homogenization assumes
infinite periodic boundary conditions and maps the response of the high-fidelity
computational model of the unit cell to an element elasticity tensor. Many works
have explored design optimization using homogenization techniques.  For example,
Andreassen, et al., in \cite{andreassen2014design} designed a manufacturable 3D
extremal elastic microstructure, achieving a material with negative Poisson’s
ratio.  In \cite{wang2019concurrent}, the authors optimize at two scales by
using a SIMP method \cite{bendsoe1989simp} to optimize multiple unit cell
structures, coupled using homogenization to a macroscale design that optimizes
the distribution of the different microstructures by a level set method.
\cite{wang2018concurrent} similarly designs a material at two scales, but using
a simpler parameterization of the unit cell and a density-based topology
optimization at the macroscale. In \cite{white2019multiscale}, a neural network
is trained to compute a homogenized elasticity tensor as a function of selected
geometric parameters of a unit cell and the microscale parameters are
incorporated in a macroscale density-based optimization. In
\cite{zhang2019graded}, homogenization is used to design layer-wise graded
lattice materials. Some hybrid methods that combine the concepts of the
homogenization of unit cell design and the control of the cross-sectional areas
of bars are developed in \cite{chen2018finite}.There are too many other works
using homogenization to compute effective material properties to describe here;
some of the research most pertinent to design of lattice-type structures
includes: \cite{watts2019simple, wang2017iga, collet2018stress}.  Although there
are many interesting works on the homogenization-based method, it cannot provide
physics information in greater detail than the finite elements used to represent
each unit cell.  Also, the infinite periodic boundary conditions do not reflect
real boundary conditions on the structure’s external surface, limiting accuracy.
The method is also limited in that the length scale of the unit cell must be
much smaller than the system length scale; otherwise, the accuracy of the method
is very low. Moreover, it is limited to a micro-structure lattice that has a
uniform configuration; e.g., that of only an octet truss. Only its relative
volume fractions vary in space.  Several of the works referenced above use
multiple unit cell structures in adjacent regions of the design domain; such a
structure violates the periodic assumption inherent to homogenization and will
also compromise solution accuracy.  Lattice structures with a uniform
configuration are prone to dislocation slips. Thus, a more accurate method for
such structures (e.g., functionally graded lattice structures, as in
\cite{zhang2019graded}) is needed.

A beam/shell element is a reduced representation for continuum solid finite
elements of a strut/plate under the assumption that the length/area of the
strut/plate is much larger than the cross-sectional area/thickness. Therefore,
the computational cost of beam/shell elements is very cheap. Because the
beam/shell element-based lattice structure design algorithms use these
simplified model of strut/plate, it is much faster than the homogenization-based
methods.  In the beam/shell element-based lattice structure design algorithms,
loads and boundary conditions are applied to the underlying design domain, and
the optimization algorithm removes unnecessary beams/shells and thickens or
shrinks the cross-sectional areas to obtain an optimal design.  Because it
starts with the user-defined design domain that is composed of many beam/shell
elements, it can directly design for macro-level lattice structures.  This
approach also allows a flexible design domain by starting with a functionally
graded lattice structures.  The beam/shell element-based method (or the ground
structure approach) was originated from Dorn in \cite{dorn1964automatic}, where
the optimal structure was a subset of a set of bars defined prior to solving the
problem.  Since the original work, many variations have been developed. For
example, Achtziger, et al., in \cite{achtziger1992equivalent} used displacement
variables with the goal to minimize compliance. Bendsoe and Ben-Tal in
\cite{bendsoe1993truss} minimized compliance for a given volume of the material
in a truss, where the mathematical model is formulated in terms of the nodal
displacements and bar cross-sectional areas, using the steepest descent
algorithm. Recently, Choi, et al., \cite{choi2019optimal} designed an optimal
lattice structure for controllable band gaps, using beam elements. Opgenoord and
Willcox \cite{opgenoord2018aeroelastic} use a beam approximation for a low-order
model of a lattice structure, combined with a nonlinear optimization of beam
areas to design additively manufactured lattice structures with desirable
aerodynamic properties.  However, it is well known that beam/shell elements have
significant issues of dealing with stress constraints. First of all, it cannot
accurately model the stress at joints.  Additionally, Kirsch in
\cite{kirsch1993fundamental} explained that the stress constraints suddenly
disappear as the cross-sectional area approaches zero, and accordingly,
degenerated feasible regions were generated.  Furthermore, the assumption of a
high length to cross-sectional area aspect ratio is often violated. Thus, the
beam/shell element-based model is fast, but inaccurate. Therefore, a new lattice
structure design algorithm that is as fast as but more accurate than the
beam/shell element-based methods is desired. Further works on the various
beam/shell element-based lattice design can be found in
\cite{achtziger1999local, hagishita2009topology, mela2014resolving} along with
two survey papers \cite{bendsoe1994optimization, stolpe2016truss}.

A good alternative to the homogenization method and beam element method in
lattice structure design problems is to use a reduced order model (ROM). Many
ROM approaches are available and have been successfully applied to various
physical simulations, such as thermostatics and thermodynamics
\cite{hoang2020domain, choi2020sns}, computational fluid dynamics
\cite{carlberg2018conservative, dal2019algebraic, choi2019space,
grimberg2020stability, kim2020fast}, large-scale transport problem
\cite{choi2020space}, porous media flow/reservoir simulations
\cite{ghasemi2015localized, jiang2019implementation, yang2016fast}, blood flow
modeling \cite{buoso2019reduced}, computational electro-cardiology
\cite{yang2017efficient}, shallow water equations \cite{zhao2014pod,
cstefuanescu2013pod}, computing electromyography \cite{mordhorst2017pod},
spatio-temporal dynamics of a predator--prey systems
\cite{dimitriu2013application}, and acoustic wave-driven microfluidic biochips
\cite{antil2012reduced}. ROMs have been also successfully applied to design
optimization problems \cite{choi2020gradient, choi2019accelerating,
amsallem2015design}. However, there are not many references that use ROMs in a
lattice structure design problems.  A few references make use of static
condensation to reduce the dimension of the structural problem to be solved; Wu,
et al., in \cite{wu2019topology} designed a hierarchical lattice structures
using super-elements. They assumed that the substructure share one common
parameterized lattice geometry pattern as in the homogenization-based method,
but instead of homogenizing a unit cell they compute the Schur complement matrix
of the substructure using a reduced basis method.  Thus, no infinite periodic
boundary condition needs to be assumed; however, the parameterization assumed is
restrictive as it is directly tied to the volume fraction of a substructure.
Therefore, the use of this method requires the design of a substructure geometry
such that the volume fraction can easily be used to adjust geometry. 


All of the above methods share a common structure: they make the solution of the
structural equations economical by adopting a surrogate model for members of the
lattice structure. Our contribution is an improved surrogate for lattice design;
we use a component-wise ROM (CWROM) based on the static
condensation reduced basis element (SCRBE) method introduced by
\cite{huynh2013static}, which builds a reduced order model for subdomains
(components) of a structure and has built-in error estimation to provide more
accuracy than current approaches. There are several advantages of the CWROM in
lattice structure design. The CWROM greatly reduces offline training costs
relative to conventional reduced order modeling approaches that require
snapshots of the full model state, because the training is done completely at
the component level. The component library built in offline training can then be
used to model any domain that can be formed by a connected set of the trained
components, so that the same offline data set may be used to explore many
different structures; for example, different kinds of functionally graded
lattices for the same part.  The CWROM also provides a large speedup while
retaining high accuracy; the numerical results find that our algorithm achieves
1,000x speedup with a less than $1\%$ relative error vs. a conforming finite
element method, which is much more accurate than the beam-based approach.
Compared to homogenization techniques, the CWROM does not rely on any
assumptions on periodicity or length scale, and has no limitations on the
geometry of the components used to form a lattice. It can also provide much more
high resolution solution information than a homogenization model, which only
recovers a solution at the level of the finite elements that model each unit
cell.  Sub-unit cell information is lost.  This level of resolution will be
useful in the context of stress-based optimization, where it is necessary to
accurately capture stress concentrations.

An additional note is required on the improvements presented by our CWROM
approach over previous work that applies static condensation with reduced order
modeling to topology optimization \cite{wu2019topology,fu2019substructuring}.
The SCRBE method exchanges some complexity in the formulation of the static
condensation equations for greater efficiency than the simpler formulations used
in the previous papers; it restricts the degrees of freedom in the problem to be
coefficients of a basis defined over the interfaces where components connect to
one another (ports), and eliminates the rest of the degrees of freedom for each
component. The formulation allows us to express the compliance objective and its
sensitivity in terms of variables in the reduced problem component space (Section
\ref{sec:cwtopopt}), and additionally allows us to make a key
simplification when used with a SIMP parameterization of the material properties
(Section \ref{sec:linear_cwrom}) that provides an additional speedup over that
given by the reduced-order model alone. Finally, there is already a body of
literature on rigorous error bounds for the SCRBE method and its extensions
\cite{huynh2013static,eftang2013pr,smetana2016optimal,smetana2015new}, allowing
certification of the designs resulting from our component-wise procedure.

The CWROM used here is not actually the SCRBE method as described in
\cite{huynh2013static}; we use the static condensation formulation from that
work, but our own parameterization allows a simplification that makes the form
of model reduction described there obsolete (Section \ref{sec:linear_cwrom}).
Instead, we apply the port reduction of Eftang and Patera \cite{eftang2013pr} to
obtain a reduced set of interface basis functions. The SCRBE approach is
originally inspired by the component mode synthesis \cite{hurty1965dynamic,
bampton1968coupling}. Many variations have been developed; for example, it is
extended to more complex problems in \cite{huynh2013complex} and to acoustic
problems in \cite{huynh2014static}.  Recently, it has been further extended to
be applicable for solid mechanics problems with local nonlinearities
\cite{ballani2018component}.  Smetana and Patera in \cite{smetana2016optimal}
proposed an optimal port spaces of the CWROM in Kolmogorov sense
\cite{kolmogoroff1936uber}. Vallaghe et al., in \cite{vallaghe2015component}
applied the CWROM approach to the parametrized symmetric eigenproblems. The 
area to which SCRBE does not apply is problems with non-localized
nonlinearities. Furthermore, it has not previously been used in topology
optimization or for lattice structure design.

A variety of other component-wise formulations have been developed, specialized
to particular applications. Buhr, et al., in \cite{buhr2017arbilomod} introduced
an adaptive component-wise reduced order model approach for fully nonlinear
problems.  Iapachino et al. \cite{iapichino2016rbm} develop a
domain-decomposition reduced basis method for elliptic problems that share
similar advantages to the CWROM used here; they also seek a reduced set of
interface basis functions, but use a different approach than
\cite{eftang2013pr}. In \cite{kaulmann2011dg}, the authors present a reduced
basis discontinuous Galerkin approach using domain decomposition for multiscale
problems. Koh et al. \cite{koh2020mmqsrv} show a reduced order TO method for
dynamic problems based on a quasi-static Ritz vector reduced basis method
applied to substructures.

Several contributions by this paper is summarized below:
\begin{itemize}
  \item A SCRBE-kind CWROM is applied to lattice--type structure design to
    accelerate the whole design optimization process.
  \item The density parameterization is chosen for the design optimization
    process and the simplification of the CWROM formulation is shown.
  \item The sensitivity of the compliance for the CWROM is derived completely in
    the component level.
  \item Error bounds for the displacement, compliance, and its sensitivity for
    the CWROM is derived. 
  \item A speedup of 1,000x and relative error of less than 1$\%$ is
    demonstrated in compliance minimization problems. 
  \item The reusability of the trained components for the lattice structure is
    demonstrated in the design optimization problems.
\end{itemize}

\subsection{Organization of the paper}\label{sec:organization}
The subsequent paper is organized as follows. We first describe the
component-wise formulation in Section~\ref{sec:cwformulation} along with
illustrations of two simple components and an example lattice that can be
constructed using these components. Section~\ref{sec:cwrom} describes the CWROM
formulation and the port reduction procedure. Section~\ref{sec:linear_cwrom}
lays out the simplification of the CWROM possible in special cases, including
the optimization formulation here.  Section~\ref{sec:cwtopopt} details a
component-wise compliance minimization problem subject to a mass constraint, and
numerical results are shown in Section~\ref{sec:results}. The paper is concluded
in Section~\ref{sec:conclusion} with summary and discussion. 


\section{Component-wise formulation}\label{sec:cwformulation}
Our component-wise full order model (CWFOM) and reduced order model (CWROM)
follow the approach explained in \cite{huynh2013static} and \cite{eftang2013port}
where static condensation (use of the Schur complement to
eliminate interior degrees of freedom) is used to eliminate the interior degrees
of freedom of each subdomain (component) and solve for only the degrees of freedom on the
interfaces where components attach (ports). It may be viewed as
an adaptation of component mode synthesis approaches [19] to provide greater
reusability of the same offline data set, but without applicability to more
complex problems where CMS succeeds. SCRBE is predicated on a decomposition of
the solution domain into subdomains, or components; each component is defined by a
parametric mapping from a reference component in an offline library.

We first describe the component-wise full order model (CWFOM), derived using
static condensation. We then describe the component-wise reduced
order model (CWROM) used in this work; this is the port-reduced static condensation
described in \cite{eftang2013port}. Finally, we present an important simplification
to the component-wise model (either the CWFOM or the CWROM) that we later apply
in our component-wise TO formulation to accelerate model evaluations beyond what
is achieved by the unmodified CWROM. For more detailed description of the component-wise formulation, we refer to
\cite{huynh2013static}.


\subsection{Component-wise FOM}\label{sec:cwfom}
The notation to describe the component-wise model inevitably becomes complex;
therefore, we provide Figures \ref{fig:ref-components} and \ref{fig:example-system}
to assist in understanding the description.
These component domains are used in our numerical examples in
Section \ref{sec:results}. The domain of the $i$-th reference component is
written $\refDomaink{i}$, and the $j$-th port on that component is indicated by
$\rportk{i,j}$. In general, a hat superscript indicates that notation refers to
a quantity in the reference domain, while the lack of one means that a quantity in
the instantiated system is intended. A component may only connect to other
components on its ports, and for the formulation of the CWFOM here to be valid,
all of the ports on a component must be mutually disjoint.

\begin{figure}[h]
\centering
  \begin{tikzpicture}
    \drawhstrutcomp{1.5}{-5}{0}
    \drawjointcomp{1.5}{+2}{(-1/sqrt(2))}
    \node at (-6.85, 0.75) {\large{$\rportk{1, 1}$}};
    \node at (-0.58, 0.75) {\large{$\rportk{1, 2}$}};
    \node at (-3.75, 0.75) {\large{$\refDomaink{1}$}};

    \draw[blue, dashed] (0, 0.7) circle [x radius = 0.25, y radius=1];
    \draw[blue, dashed] (-7.5, 0.7) circle [x radius = 0.25, y radius=1];

    \node at (2.55, 0.75) {\large{$\rportk{2, 1}$}};
    \node at (5, 0.75) {\large{$\rportk{2, 3}$}};
    \node at (3.75, 2.1) {\large{$\rportk{2, 2}$}};
    \node at (3.75, -0.55) {\large{$\rportk{2, 4}$}};
    \node at (3.75, {(1/sqrt(2))}) {\large{$\refDomaink{2}$}};

    \draw[blue, dashed] (1.95, 0.7) circle [x radius = 0.25, y radius = 1];
    \draw[blue, dashed] (5.6, 0.7) circle [x radius = 0.25, y radius = 1];
    \draw[blue, dashed] (3.75, 2.6) circle [x radius = 1, y radius = 0.25];
    \draw[blue, dashed] (3.75, {(-1.5/sqrt(2))}) circle [x radius = 1, y radius = 0.25];
  \end{tikzpicture}
\caption{Two reference components for a 2D lattice structure. Reference domains
  are indicated by $\refDomaink{1}$ and $\refDomaink{2}$ for joint and strut
  components, respectively.  Four local ports are indicated by
  $\rportk{1,\portindex}$, $\portindex\in\nat{4}$ for the joint component, while
  two local ports are indicated by $\rportk{2,\portindex}$,
  $\portindex\in\nat{2}$ for the strut component.}
\label{fig:ref-components}
\end{figure}
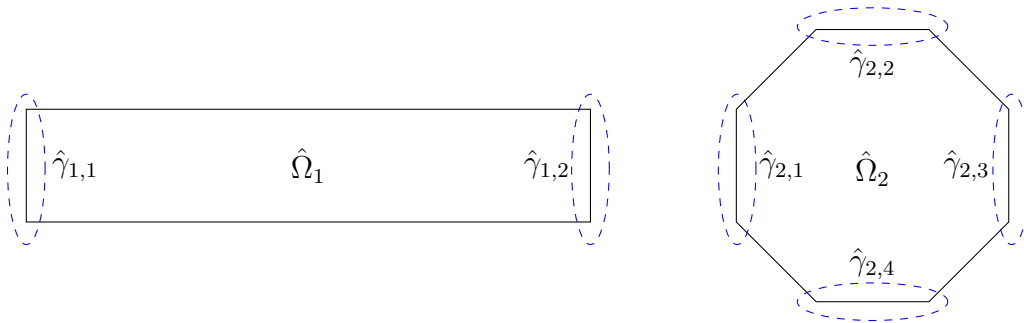

\begin{figure}[h]
  \centering
  \begin{tikzpicture}
    \smallsystem{0.7}{0}{0}
    \node at (0.4, 3.5) {$\sysDomaink{1}$};
    \node at (0.4, 1.9) {$\portSymbol_1$};
    \node at (0.4, 5) {$\portSymbol_2$};
    \node at (5.55, 3.5) {$\sysDomaink{5}$};
    \node at (5.55, 5) {$\portSymbol_9$};
    \node at (5.55, 1.9) {$\portSymbol_{10}$};
    
    \node at (0.4, 6) {$\sysDomaink{2}$};
    \node at (-0.7, 6) {$\portSymbol_{3}$};
    \node at (0.4, 7.05) {$\portSymbol_{4}$};
    \node at (3, 6) {$\sysDomaink{3}$};
    \node at (1.4, 6) {$\portSymbol_{5}$};
    \node at (4.5, 6) {$\portSymbol_{6}$};
    \node at (5.55, 6) {$\sysDomaink{4}$};
    \node at (5.55, 7.05) {$\portSymbol_{7}$};
    \node at (6.6, 6) {$\portSymbol_{8}$};

    \node at (5.55, 0.8) {$\sysDomaink{6}$};
    \node at (6.65, 0.8) {$\portSymbol_{11}$};
    \node at (5.55, -0.2) {$\portSymbol_{12}$};
    \node at (3, 0.8) {$\sysDomaink{7}$};
    \node at (4.4, 0.8) {$\portSymbol_{13}$};
    \node at (1.47, 0.8) {$\portSymbol_{14}$};
    \node at (0.4, 0.8) {$\sysDomaink{8}$};
    \node at (0.4, -0.2) {$\portSymbol_{15}$};
    \node at (-0.8, 0.8) {$\portSymbol_{16}$};

    \drawvstrutcomp{1}{8}{1}
    \drawhstrutcomp{1}{10}{4}

    \node at (8.5, 3.5) {$\sysDomaink{5}$};
    \node at (8.5, 5.8) {$\portSymbol_{9}$};
    \node at (8.5, 1.2) {$\portSymbol_{10}$};
    \node at (12.5, 4.5) {$\refDomaink{1}$};
    \node at (10.5, 4.5) {$\rportk{1, 1}$};
    \node at (14.6, 4.5) {$\rportk{1, 2}$};

    \draw[->,red,thick] (12.5, 3.9) to[bend left] (9.1, 2);
    \node at (12.2, 2.5) {$\transMapk{5}(\cdot; \paramk{5})$};

    \draw[blue, thick, dashed] (5.55, 3.5) circle [x radius = 0.7, y radius = 2.3];
    \draw[blue, thick, ->] (6.3, 4) to[bend left] (7.9, 4.5);
  \end{tikzpicture}
  \caption{An example of a 2D lattice structure that assembled from several
  instances of two reference components in Fig.~\ref{fig:ref-components}. four
  struts and four joints form this particular lattice structure; thus the number
  of instantiated components, $\numInsta$, is 8. Each component is assigned a
  parameter vector, $\paramk{i}$, that defines both the geometry and other
  properties of the instantiated component. Each instantiated component is related
  to a reference component by the transformation mapping $\transMapk{i}$; this
  relationship is illustrated on the right. Note that this transformation may
  include parameters that do not affect the geometry; e.g., physical properties
  such as Young's modulus.}
  \label{fig:example-system}
\end{figure}
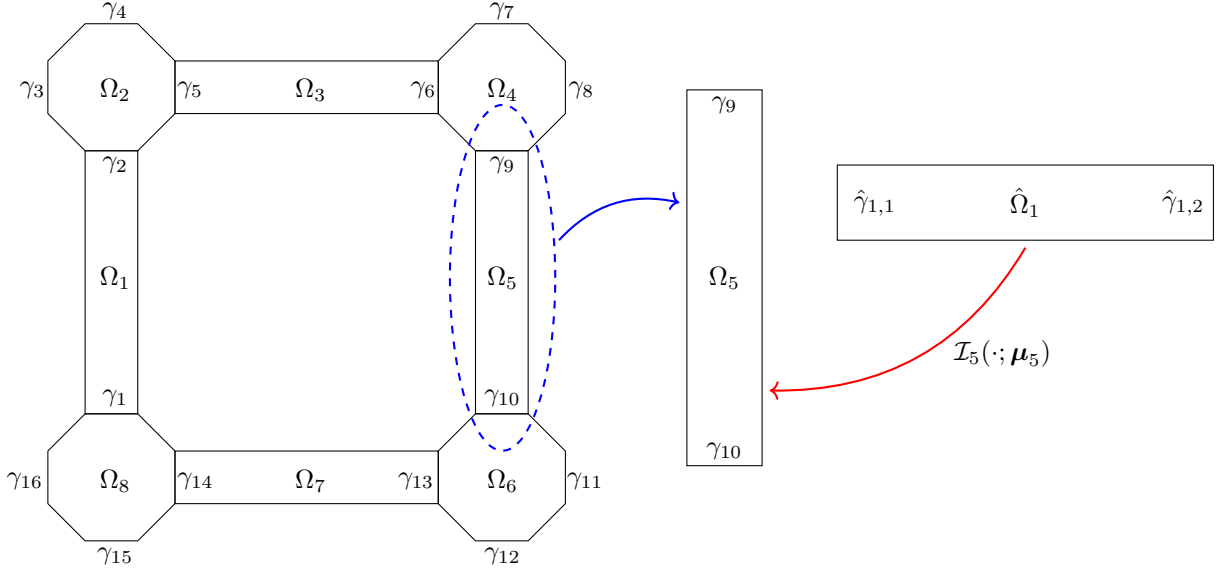

Note that there are infinitely many lattice systems that can be built using the
same two reference components in this manner; the variation is only limited by
the domain of the transformation map $\transMapk{i}$ and geometric compatibility.
This property makes the CWROM uniquely powerful because it can be used to model
many domains in the online phase while only training the reduced order models on
the reference component domains.

A discrete finite element system-level problem states that, for any parameter
$\param\in\paramDomain$ and a given system domain $\sysDomain$, the approximated
solution $\sol(\param) \in \femSpace(\sysDomain)$ satisfies
\begin{equation}\label{eq:weakform}
  \bilinear{\sol(\param),\testfun;\param} = \linear{\testfun; \param},
  \hspace{6pt} \forall \testfun\in\femSpace(\sysDomain),  
\end{equation}
where $\femSpace(\sysDomain) \subset \solSpace(\sysDomain)$ denotes
the discrete system finite element space, $\testfun \in \femSpace(\sysDomain)$
denotes a test function, $\bilinearSymbol: \HH \times \HH \times \paramDomain
\mapto \RR{}$ is a symmetric, coercive, bilinear form, and $\linearSymbol: \HH\times
\paramDomain \mapto \RR{}$ denotes a linear form.

In component-wise formulation, the physical system domain $\sysDomain$ and the
system parameter domain $\paramDomain$ are decomposed as $\sysDomainClosure =
\cup_{\instaindex=1}^{\numInsta} \sysDomainClosurek{\instaindex}$ and
$\paramDomain \subseteq \oplus_{\instaindex=1}^{\numInsta}
\paramDomaink{\instaindex}$, respectively,  
where $\numInsta$ is the number of decomposed components.
We also decompose the system parameter as an $\numInsta$-tuple: $\param =
(\paramk{1},\ldots,\paramk{\numInsta})$.  

Each decomposed component is mapped from a corresponding reference component in a
library of reference components containing $\numRef$ reference components.
In the 2D lattice system in Fig. \ref{fig:example-system}, we have two components:
a joint and a strut, so $\numRef = 2$. Each reference component has an
associated reference domain $\refDomaink{\refindex}$, $\refindex \in
\nat{\numRef}$, $\nat{\numRef} \equiv \{1,\ldots,\numRef\}$.  Each reference
component domain boundary is denoted as $\boundary{\refDomaink{\refindex}}$ and
it has a set of $n_\refindex^\gamma$ disjoint reference ports,
denoted as $\rportk{r,j},\ j \in \nat{n_\refindex^\gamma}$. For example, the
joint component has $n_2^\gamma = 4$, while the strut component has $n_1^\gamma
= 2$ as shown in Figure \ref{fig:ref-components}. A system is formed by
$\numInsta$ instantiated reference components from the library.
For example, Fig. \ref{fig:example-system} shows eight
instantiated component domains connected through ports to form a
two-dimensional lattice system where each component is mapped from one of the
reference components in Fig. \ref{fig:ref-components}. Note that a port can be
used either to connect two instantiated components or to serve as a boundary.
Each reference component is associated with a discrete finite element space
$\femSpacek{\refindex},\refindex \in \nat{\numRef}$; we also denote the dimension
of this space as $\spaceDimension{r}{h}$, $r \in \nat{\numRef}$. The port space
$\portSpacek{r,j}$ of dimension $\spaceDimension{r,j}{\portSymbol}$ is defined
as the restriction of $\femSpacek{\refindex}$ to $\rportk{r,j},j\in
\nat{n_r^\gamma}$. The formulation of the problem in reference domain finite
element spaces is important to the performance of the component-wise method;
the computations in the instantiated system are performed in the reference
domain through the transformation mapping, defined below.

We require several mappings for bookkeeping between a reference component of the
library and each instance in the system. First, the indices for the instance-reference
pair can be tracked by a mapping $\compIndexMap:\nat{\numInsta}\to\nat{\numRef}$
that maps each of the $\numInsta$ component instantiations to one of the
$\numRef$ reference components in the library (e.g., in
Fig.~\ref{fig:example-system}, $\compIndexMapk{4}=2$ and $\compIndexMapk{7}=1$).
Note that the $i$-th instantiated component may connect
to at most $n_{\compIndexMapk{i}}^\gamma$ other instantiated components in the
system through its local ports. The discrete finite element space,
$\femSpace(\sysDomain)$, can now be also decomposed and defined as a direct sum
of component finite element spaces: $\femSpace(\sysDomain) =
\left(\oplus_{i=1}^{\numInsta} \femSpacek{\compIndexMapk{i}}\right) \cap
X(\sysDomain)$; hence $\femSpace(\sysDomain)$ inherits the boundary conditions
and global continuity enforced by the continuous finite element space,
$X(\sysDomain)$. Now, we define the transformation map, $\transMapk{i} :
\refDomaink{\compIndexMapk{i}} \times \paramDomaink{i} \to \sysDomain_i$ that connects the instantiated
component domain with the corresponding reference domain as $\sysDomain_i \equiv
\transMapk{i}\left(\refDomaink{\compIndexMapk{i}}; \paramk{i}\right)$.
It follows naturally that the transformation map defines the connection between
the instantiated component local port, $\portSymbol_{i,j}$ (the $j$-th port
on the $i$-th instantiated component) with a reference component. For example,
Fig. \ref{fig:example-system} illustrates the transformation map
$\transMapk{i}: \refDomaink{\compIndexMapk{i}} \to \sysDomain_i$ for $i = 5$,
where the transformation maps the reference port $\rportk{1,1}$ to the
instantiated port $\gamma_{9}$ and $\rportk{1, 2}$ to $\gamma_{10}$.

A local-to-global port mapping $\globpmapk{i}: \nat{n_{\compIndexMapk{i}}^\gamma} \to \nat{n_0^\gportSymbol}$
maps a local port index to a global port index. Here, $n_0^\gportSymbol$ denotes
the number of global ports in the system. We also denote the number of global
ports excluding ports on which Dirichlet boundary conditions are applied
by $n^\gportSymbol$. The global port index $p \in
\nat{n_0^\gportSymbol}$ is obtained from
a local port $j$ on component $i$ in the system by $p = \globpmapk{i}(j)$;
that is, the global port $\gamma_p$ is the same port as the local port $\gamma_{i,j}$.
For example, in Fig. \ref{fig:example-system}, we see that $\gamma_4
= \gamma_{2,2}$ and $\gamma_1 = \gamma_{8,2} = \gamma_{1, 2}$.
The connectivity of the system is defined through index sets
$\pi_p, p \in \nat{n_0^\gportSymbol}$. In the case of
an interior global port (coincidence of two local ports $\gamma_{i,j}$ and
$\gamma_{i',j'}$), we set $\pi_p = \left\{(i,j),(i',j')\right\}$; and in the
case of a boundary global port (a single local port $\portSymbol_{i,j}$), we set
$\pi_p = \left\{(i,j)\right\}$. For example, in Fig.
\ref{fig:example-system}, we see that $\pi_2 =
\left\{(1,1),(2,4)\right\}$ and $\pi_3 = \left\{(2, 1)\right\}$.

Using the maps defined above, the bilinear and linear forms can also be
decomposed; for any $w, v \in \HH(\sysDomain)$,
\begin{equation}
  \label{eq:decomposed_bilinear}
  \bilinear{w, v; \param} = \sum_{i=1}^{\numInsta} a_{\compIndexMapk{i}}
  \left(w|_{\sysDomain_i},v|_{\sysDomain_i};\paramk{i}\right)
\end{equation}
and
\begin{equation}
  \label{eq:decomposed_linear}
  f(v;\param) = \sum_{i=1}^{\numInsta}f_{\compIndexMapk{i}}
  \left(v|_{\sysDomain_i}\right)
\end{equation}
The port space for each reference port, $\portSpacek{r,j}$, is defined by its basis:
\begin{equation}\label{eq:port_space}
  \portSpacek{r,j} \equiv \Span{\portbasis{r,j,k},
  r\in\nat{\numRef},j\in\nat{n_r^\gamma},k\in\nat{\mathcal{N}_{r,j}^\gamma}},
\end{equation}
where $\portbasis{r, j, k}$ are linearly independent; for the model to be full order,
$\portSpacek{r,j}$ must equal the restriction of $\femSpacek{r}$ to $\rportk{r,j}$.
As a compatibility condition to enforce continuity, we require that for any global port index
$\pi_p = \left\{(i,j),(i',j')\right\}$,
\begin{equation}
  \portbasis{\compIndexMapk{i},j,k} =
  \portbasis{\compIndexMapk{i'},j',k},
  k \in \nat{\mathcal{N}_p^\gamma}
\end{equation}
Here, the full dimension of the port is either
$\mathcal{N}_p^\gamma$ or $\mathcal{N}_{r,j}^\gamma, r = \compIndexMapk{i}$,
depending on whether global or local ports are used; here, $(\compIndexMapk{i}, j) \in \pi_p$.
Given the reference port bases $\chi_{r, j, k}$, we additionally define an extension
to the interior of a reference component, $\psi_{r, j, k}$, which is equal to $\chi_{r, j, k}$
on $\rportk{r, j}$, zero on the other ports, and varies smoothly in the interior.
For example, in \cite{huynh2013static}, the Laplacian lifted $\psi_{r, j, k}$ are defined
by
\begin{equation}\label{eq:laplace_lifting}
  \begin{split}
    \Delta \psi_{r, j, k} = 0 \ \text{ in } \refDomaink{r} \\
    \psi_{r, j, k} = \chi_{r, j, k} \ \text{ on } \rportk{r, j} \\
    \psi_{r, j, k} = 0 \ \text{ on } \rportk{i, j}, \ i \neq r,
  \end{split}
\end{equation}
however, other choices of lifting are possible, and even desirable; see Section
\ref{sec:linear_cwrom}.

We may now express the global solution as
\begin{equation}
  \label{eq:scrbe_sol}
  \sol(\param) = \sum_{i=1}^{\numInsta} \bub{i}{f;h}(\paramk{i})
  + \sum_{p=1}^{n^\gamma} \sum_{k=1}^{\mathcal{N}_p^\gamma}
  \cwSolk{p,k}(\param)\Phi_{p,k}^h(\param)
\end{equation}
where all the terms except $\cwSolk{p,k}(\param)$ can be obtained through
component-wise computations. For example, $\bub{i}{f;h}(\paramk{i}) \in
\bubspace{\compIndexMapk{i};0}{h}$ is a bubble function associated with the
component right-hand side, which satisfies
\begin{equation}
  \label{eq:forcing_bubble}
  a_{\compIndexMapk{i}}\left(\bub{i}{f;h}(\paramk{i}), v; \paramk{i}\right)
  = f_{\compIndexMapk{i}}\left(v; \paramk{i}\right), \ \forall v
  \in \bubspace{\compIndexMapk{i};0}{h}
\end{equation}
where $\bubspace{r;0}{h}, r \in \nat{\numRef}$ are the
\textit{bubble} spaces associated with each reference component domain by
\begin{equation}
  \bubspace{r;0}{h} \equiv
  \left\{
    w \in \femSpacek{r}: w|_{\rportk{r,j}} = 0, j \in \nat{n_r^\gamma}
  \right\}, \ r \in \nat{\numRef}
\end{equation}

In order to define the patched interface basis functions $\Phi_{p,k}^h(\param)$,
we first need to define the interface function $\phi_{i,j,k}^h(\paramk{i})$ as
\begin{equation}
  \label{eq:ifunc}
  \phi_{i,j,k}^h(\paramk{i}) \equiv \bub{i,j,k}{h}(\paramk{i}) +
  \psi_{\compIndexMapk{i},j,k},
\end{equation}
where the interface bubble functions $\bub{i,j,k}{h}(\paramk{i}) \in
\bubspace{\compIndexMapk{i};0}{h}, k \in
\nat{\spaceDimension{\compIndexMapk{i},j}{\gamma}}, j \in
\nat{n_{\compIndexMapk{i}}^\gamma}$, satisfy
\begin{equation}
  \label{eq:ifunc_bubble}
  a_{\compIndexMapk{i}}\left(\bub{i,j,k}{h}(\paramk{i}), v; \paramk{i}\right) =
  -a_{\compIndexMapk{i}}\left(\psi_{\compIndexMapk{i},j,k}, v; \paramk{i}\right),
  \ \forall v \in \bubspace{\compIndexMapk{i};0}{h}
\end{equation}

The patched interface basis function $\ifunc{p,k}(\param)$ for $\pi_p =
\left\{(i,j),(i',j')\right\}$ is defined as $\ifunc{p,k}\equiv \phi_{i,j,k}^h +
\phi_{i',j',k'}^h$, while we define $\ifunc{p,k} \equiv \phi_{i,j,k}^h$ for a
boundary global port $\pi_p = \left\{(i,j)\right\}$. All the patched interface
functions and the bubble functions are extended by zero outside of their
associated component so that the global solution representation (\ref{eq:scrbe_sol})
makes sense.

We now plug Eq. (\ref{eq:scrbe_sol}) into (\ref{eq:weakform}) and note that the
only unknowns are $\cwSolk{p,k}(\param)$, i.e., the coefficients for the patched
interface basis functions, after the component-wise computations for the bubble
functions and the patched interface basis functions. Therefore, we only need to
set the test functions to be active on the \textit{skeleton}, whose space
$\skeletonSpace$ is defined as
\begin{equation}\label{eq:skeleton}
  \skeletonSpace \equiv \Span{\ifunc{p,k}(\param),
  p\in \nat{n^\gamma}, k \in \nat{\spaceDimension{p}{\gamma}}}
  \subset X^h
\end{equation}
We denote the number of unknowns as $n_{SC} = \sum_{p=1}^\gamma
\spaceDimension{p}{\gamma}$. The weak form to solve for the unknowns,
$\cwSolk{p,k}(\param)$, can be equivalently written as the following linear
algebraic system of equations, i.e., for any $\param \in \paramDomain$, find
$\cwSol(\param) \in \RR{n_{SC}}$ such that
\begin{equation}
  \label{eq:cw_linear_system}
  \cwStiffness(\param)\cwSol(\param) = \cwForce(\param)
\end{equation}
where
\begin{equation}\label{eq:cwStiffness}
  \cwStiffness_{(p,k),(p',k')}(\param) =
  \bilinear{\ifunc{p,k},\ifunc{p',k'};\param}
\end{equation}
\begin{equation}\label{eq:cwForce}
  \cwForce_{(p,k)}(\param) = \linear{\ifunc{p,k}(\param); \param} - 
  \sum_{i=1}^{\numInsta} \bilinear{\bub{i}{f;h}(\paramk{i}), \ifunc{p,k}(\param);
  \param}
\end{equation}
for $p,p' \in \nat{n^\gamma}, k \in \nat{\spaceDimension{p}{\gamma}},$ and
$k' \in \nat{\spaceDimension{p'}{\gamma}}$. Note that $(p,k)$ is a double-index
notation for a single degree of freedom. The assembly of $\cwStiffness(\param)$
and $\cwForce(\param)$ can be done by looping over the local Schur complement
matrices and load vectors (Eqs. \ref{eq:local_schur} and \ref{eq:local_load})
according to Algorithm \ref{alg:sc_assemble}.
\begin{equation}\label{eq:local_schur}
  \cwStiffness_{(j,k),(j',k')}^{i}(\paramk{i}) =
  a_{\compIndexMapk{i}}\left(\phi_{i,j,k}^h(\paramk{i}),
  \phi_{i,j',k'}^h(\paramk{i});\paramk{i}\right)
\end{equation}
\begin{equation}\label{eq:local_load}
  \cwForce_{(j,k)}^i(\paramk{i}) =
  f_{\compIndexMapk{i}}\left(\phi_{i,j,k}^h(\paramk{i});\paramk{i}\right) -
  a_{\compIndexMapk{i}}\left(\bub{i}{f;h}(\paramk{i}), \phi_{i,j,k}^h(\paramk{i});
  \paramk{i}\right)
\end{equation}
\begin{algorithm}
  \begin{algorithmic}
    \STATE $\cwForce_0(\param) = \bm{0}, \cwStiffness_0(\param) = \bm{0}$
    \FOR {$i = 1,\ldots,\numInsta$}
      \FOR {$j=1,\ldots,n_{\compIndexMapk{i}}^\gamma$}
        \FOR {$k = 1,\ldots,\spaceDimension{\compIndexMapk{i},j}{\gamma}$}
          \STATE $\cwForce_{0;\globpmapk{i}(j),k}(\param) \gets
            \cwForce_{0;\globpmapk{i}(j),k}(\param) +
            \cwForce_{(j,k),(j',k')}^i(\paramk{i})$
          \FOR {$j' = 1,\ldots,n_{\compIndexMapk{i}}^\gamma$}
            \FOR {$k' = 1,\ldots,\spaceDimension{\compIndexMapk{i},j'}{\gamma}$}
              \STATE $\cwStiffness_{0;\left(\globpmapk{i}(j),k\right),
                \left(\globpmapk{i}(j'),k'\right)}(\param) \gets
                \cwStiffness_{0;\left(\globpmapk{i}(j),k\right),
                \left(\globpmapk{i}(j'),k'\right)}(\param) +
                \cwStiffness_{(j,k),(j',k')}^i(\paramk{i}) $
            \ENDFOR
          \ENDFOR
        \ENDFOR
      \ENDFOR
    \ENDFOR
    \STATE {Eliminate port Dirichlet degrees of freedom:
    $\cwForce_0(\param) \to \cwForce(\param)$ and
    $\cwStiffness_0(\param) \to \cwStiffness(\param)$}
  \end{algorithmic}
  \caption{Component-based static condensation assembly loop}
  \label{alg:sc_assemble}
\end{algorithm}
\begin{remark}
  The Schur complement matrix $\cwStiffness(\param)$ is symmetric and positive-definite
  (SPD), thanks to symmetry and coercivity of $\bilinear{\cdot,\cdot; \param}$, the
  definition of $\cwStiffness(\param)$ in (\ref{eq:cwStiffness}) and linear
  independence of the $\ifunc{p,k}(\param)$, $k \in \nat{\spaceDimension{p}{\gamma}}$,
  $p \in \nat{n^\gamma}$.
\end{remark}


\subsection{Component-wise reduced order model}\label{sec:cwrom}
The CWFOM presented in Section \ref{sec:cwfom} reduces the number of degrees of
freedom in the original problem, but in the general case requires more work to
construct the linear system than would be needed to solve a finite element model.
Its computational cost may be reduced by introducing reduced port bases
\cite{eftang2013port}. We denote
the component-wise method with reduced port bases the component-wise ROM
(CWROM). The port reduction is effected by introducing a subspace
$\portSubspacek{i,j}$ of dimension $\reducedSpaceDimension{i,j}{\gamma} <
\spaceDimension{\compIndexMapk{i},j}{\gamma}$ of the port space
$\portSpacek{\compIndexMapk{i},j}$ defined in (\ref{eq:port_space}).  For
example, we define the reduced port space $\portSubspacek{\compIndexMapk{i},j}$ for instantiated
component $i$ and port index $j$, as
\begin{equation}\label{eq:reduced_port_space}
  \portSubspacek{\compIndexMapk{i},j} \equiv
  \Span{\portbasis{\compIndexMapk{i},j,k}, i \in \nat{\numInsta}, j \in
  \nat{n_{\compIndexMapk{i}}^\gamma}, k \in
  \nat{\reducedSpaceDimension{\compIndexMapk{i},j}{\gamma}}
  }.
\end{equation}

As in the CWFOM, we impose $\reducedSpaceDimension{\compIndexMapk{i},j}{\gamma} =
\reducedSpaceDimension{\compIndexMapk{i'},j'}{\gamma}$ and
$\portbasis{\compIndexMapk{i},j,k} = \portbasis{\compIndexMapk{i'},j',k}$
on global port $\pi_p =
\left\{(i,j),(i',j')\right\}$ for solution continuity, which will be
satisfied naturally by the pair-wise training approach. We denote the number of
reduced global port degrees of freedom as $\reducedSpaceDimension{p}{\gamma}$
where we must have $\reducedSpaceDimension{p}{\gamma} =
\reducedSpaceDimension{\compIndexMapk{i},j}{\gamma}= \reducedSpaceDimension{\compIndexMapk{i'},j'}{\gamma}$ for
$\pi_p = \left\{(i,j),(i',j')\right\}$.  Then, the global solution to the CWROM
can be expressed as
\begin{equation}\label{eq:reduced_sol}
  \reducedSolution(\param) = \sum_{i=1}^{\numInsta} \bub{i}{f;h}(\paramk{i}) +
  \sum_{p=1}^{n^\gamma} \sum_{k=1}^{\reducedSpaceDimension{p}{\gamma}}
  \reducedCWsolk{p,k}(\param)\rifunc{p,k}(\param)
\end{equation}
where $\bub{i}{f;h}(\paramk{i})$ is obtained and defined as in
(\ref{eq:forcing_bubble}). The reduced patched interface functions
$\rifunc{p,k}(\param)$ can be obtained by following the same procedure introduced
in Section \ref{sec:cwfom}; first, the reduced lifted port basis $\tilde{\psi}_{r,j,k}$
is obtained by lifting the members of $\portSubspacek{r,j}$, e.g. using the lifting
in Eq. \ref{eq:laplace_lifting}. Then the reduced interface functions $\tilde{\phi}^h_{r,j,k}(\param)$
are defined in terms of $\tilde{\psi}_{r,j,k}$ just as shown in Eqs. \ref{eq:ifunc} and
\ref{eq:ifunc_bubble}, and finally, the reduced patched interface functions are defined by
$\rifunc{p,k}(\param) = \tilde{\phi}^h_{i,j,k}(\param) + \tilde{\phi}^h_{i',j',k}(\param)$
where $\pi_p = \left\{(i,j),(i',j')\right\}$, or $\rifunc{p,k}(\param) = \tilde{\phi}^h_{i,j}(\param)$
for $\pi_p = \left\{(i,j)\right\}$. We introduce the reduced coordinate,
$\reducedCWsolk{p,k}$ to distinguish it from the coordinate $\CWsolk{p,k}$ in
\eqref{eq:scrbe_sol}.  Therefore, it is key to build a good reduced port space
$\portSubspacek{i,j}$.

This model reduction reduces both the eventual size of the Schur complement
system to be solved and the cost of its construction; the latter is because the
number of interface bubble functions to be solved for from Eq. \ref{eq:ifunc_bubble}
on reference port $\rportk{i,j}$ is reduced to the number of elements of the reduced
port basis $\portSubspacek{i,j}$. It is key to build a reduced port space that
captures the full range of behavior of solutions for all systems in which a component
will be instantiated. There are many ways to construct $\portSubspacek{i,j}$; for
example, any orthogonal polynomials can serve as a basis of the reduced port
space if the solution is assumed to be smooth on the port, such as Legendre or
Chebyshev polynomials for 1D ports and Zernike polynomials for 2D unit disc
ports.  These are special cases of Gegenbauer
polynomials, thus a special type of Jacobi polynomials.  Other types, such as
Wilson or Askey-Wilson polynomial types, may also serve. In order to achieve a
port basis that captures the behavior of the solution for all instantiated
systems, however, Eftang and Patera introduce a pairwise training procedure in
\cite{eftang2013port} in which each port space is constructed empirically by
considering all possible connections between two components. The procedure is
described in Algorithm \ref{alg:pairwise_training} and illustrated in Figure
\ref{fig:pairwise}.

\begin{algorithm}
  \hspace*{\algorithmicindent}\textbf{Input:} Two component domains 
  $\Omega_{1}$ and $\Omega_{2}$ connected at a common port
  $\gamma_{p^*} = \rportk{\compIndexMapk{1},j} = \rportk{\compIndexMapk{2}, j'}$.

  \hspace*{\algorithmicindent}\textbf{Output:} $S_{pair} \neq \emptyset$
  \begin{algorithmic}
    \FOR {$i = 1, \ldots,N_{samples}$}
      \STATE {Assign random parameters $\paramk{1} \in \paramDomaink{1}$ and
      $\paramk{2} \in \paramDomaink{2}$ to the two components}
      \STATE {On all non-shared ports $\gamma_p, p \neq p^*$, assign random
      boundary conditions:
        \begin{equation}
          u|_{\gamma_p} = \sum_{k=1}^{n_{i,j}^\gamma} \frac{q}{k^\eta}
          \legendre{i,j}
        \end{equation}
      where $\pi_p = \left\{(i,j)\right\}$
      }
      \STATE {Solve the governing equation (\ref{eq:weakform}) on the
      two-component system}
      \STATE {Extract solution $u_{\gamma_{p^*}}$ on the shared port}
      \STATE {Add mean-corrected port solution to snapshot set:
        \begin{equation}
          S_{pair} \gets S_{pair} \cup \left(u_{\gamma_{p^*}} -
          \frac{1}{\left|\gamma_{p^*}\right|}
          \int_{\gamma_{p^*}}u|_{\gamma_{p^*}}\right)
        \end{equation}
      }
    \ENDFOR
  \end{algorithmic}
  \caption{Pairwise training for reduced port spaces}
  \label{alg:pairwise_training}
\end{algorithm}

\begin{figure}[h]
  \centering
  \begin{tikzpicture}
    \drawhstrutcomp{1.4}{0}{0}
    \drawjointcomp{1.4}{(5+1/sqrt(2))}{(-1/sqrt(2))}

    \node at ({(2.5*1.4)}, 0.7) {\large{$\refDomaink{1}$}};
    \node at ({(1.4*(5.5+1/sqrt(2)))}, 0.7) {\large{$\refDomaink{2}$}};

    \node at (0.28, 0.7) {$\gamma_2$};
    \node at ({(4.8*1.4)}, 0.7) {$\gamma_1$};
    \node at ({1.4*(6+sqrt(2)-0.2)}, 0.7) {$\gamma_4$};
    \node at ({1.4*(5.5+1/sqrt(2))}, {1.4*(1 + 1/sqrt(2)-0.2)}) {$\gamma_3$};
    \node at ({1.4*(5.5+1/sqrt(2))}, {1.4*(-1/sqrt(2)+0.2)}) {$\gamma_5$};
  \end{tikzpicture}
  \caption{Illustration of the pairwise training for the shared port $\gamma_1$ on 
    reference component domains
    $\refDomaink{1}$ and $\refDomaink{2}$. First of all, randomly assign
    $\paramk{1}\in\paramDomaink{1}$ and $\paramk{2}\in\paramDomaink{2}$.
    Second, assign random boundary conditions on all non-shared ports, i.e.,
    $\gamma_{2}$, $\gamma_3$, $\gamma_4$, and $\gamma_5$ as shown
    in Step 3 of Algorithm \ref{alg:pairwise_training}.  Third, extract the solution of the governing
    equation \eqref{eq:weakform} on the shared port, i.e., $\sol|_{\gamma_1}$
    and add it to the snapshot sets after subtracting the average as shown in Step
    6 of Algorithm \ref{alg:pairwise_training}. Repeat this process
    $\numSamples$ times.}
    \label{fig:pairwise}
\end{figure}
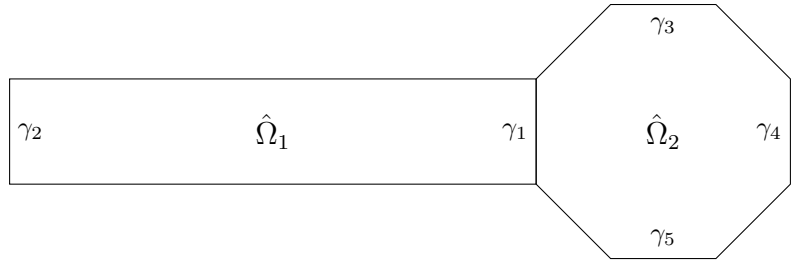

To describe the procedure, we first define the discrete generalized Legendre
polynomials, $\legendre{r,j}$, for port $j$ of the reference component $r$,
which satisfy the singular Sturm-Liouville eigenproblem:
\begin{equation}
  \label{eq:sturm-liouville}
  \int_{\rportk{r,j}} s_{r,j} \nabla \legendre{r,j} \cdot \nabla \testfun =
  \Lambda_{r,j}^k \int_{\rportk{r,j}} \legendre{r,j}v, \ \forall v \in
  P_{r,j}^h, \ k \in \nat{\spaceDimension{r,j}{\gamma}}
\end{equation}
Here, the port boundary vanishing diffusion mode, $s_{r,j} \in P_{r,j;0}^h$, can
be obtained by solving
\begin{equation}
  \label{eq:vanishing_mode}
  \int_{\rportk{r,j}} \nabla s_{r,j} \cdot \nabla v = \int_{\rportk{r,j}} v,
  \ \forall v \in P_{r,j;0}^h
\end{equation}
where the port space with homogeneous boundary $P_{r,j;0}^h$ is defined as
$P_{r,j;0}^h \equiv \left\{v \in P_{r,j}^h : v_{\partial\rportk{r,j}} =
0\right\}$.  Note that $\partial\rportk{r,j}$ describes two end points in 1D
ports and boundaries (curves) in 2D ports. The Legendre polynomials,
$\legendre{r,j}$, are used to specify random boundary conditions $u_{\gamma_p}$
in the pairwise training procedure (Step 3 of Algorithm
\ref{alg:pairwise_training}). There, the random variable, $q \in \RR{}$, is
drawn from a univariate uniform or log uniform distribution over $(-1, 1)$ and
the tuning parameter, $\eta \geq 0$, acts as a control for the expected regularity of
solutions. Then the governing
equation (\ref{eq:weakform}) on the pair of two components is solved and the
solution on the shared port, $u_{\gamma_{p^*}}$, is extracted. In Step 6, we
subtract the extracted solution's average and add to the set of snapshots
$S_{pair}$. This ensures that the resultant port basis is orthogonal to the
constant function. This procedure should be repeated for a reference port
$\gamma_{i,j}$ for each configuration of two instantiated components in which
the port will be used in the online phase, to ensure that the reduced space
constructed captures the solution well in all configurations.

Once the snapshot set $S_{pair}$ is constructed, the
$\reducedSpaceDimension{i,j}{\gamma}$ port basis vectors can be found by the
proper orthogonal decomposition (POD).  The basis from POD is an optimally
compressed representation of $\Span{\bm{S}_{pair}}$ in the sense that it
minimizes the difference between the original snapshot matrix and the projected
one onto the port subspace $\hat{P}_{i,j}^{h}$:
\begin{equation}
  \label{eq:pod}
  \minimize_{\bm{\chi}\in\RR{\spaceDimension{i,j}{\gamma} \times
  \reducedSpaceDimension{i,j}{\gamma}}, \bm{\chi}^T\bm{\chi} = \bm{I}_{
    \reducedSpaceDimension{i,j}{\gamma} \times \reducedSpaceDimension{i,j}{\gamma}
  }}
  \left\|\bm{S}_{pair} - \bm{\chi}\bm{\chi}^T\bm{S}_{pair}\right\|_F^2
\end{equation}
where $\left\|\cdot\right\|_F$ denotes the Frobenius norm and $\bm{S}_{pair} \in
\RR{\spaceDimension{i,j}{\gamma} \times l}$ denotes a matrix whose columns
consist of the mean-corrected port solutions in Step 6 of Algorithm
\ref{alg:pairwise_training}, $l$ denotes the number of snapshots in $S_{pair}$
and $\bm{\chi} \in \RR{\spaceDimension{i,j}{\gamma}
\times\reducedSpaceDimension{i,j}{\gamma}}$ is the port basis matrix, i.e.,
$\hat{P}_{i,j}^h = \Span{\bm{\chi}}$, which play the role of unknowns in the
minimization problem (\ref{eq:pod}). The solution of this minimization can be
obtained by setting $\bm{\chi}$ as the first
$\reducedSpaceDimension{i,j}{\gamma}$ columns of $\bm{U}$, where $\bm{U}$ is the
left singular matrix of the following thin singular value decomposition (SVD):
\begin{equation}
  \label{eq:svd}
  \bm{S}_{pair} = \bm{U\Sigma V}^T
\end{equation}
where $\bm{U} \in \RR{\spaceDimension{i,j}{\gamma} \times l}$ and $\bm{V} \in
\RR{l\times l}$ are orthogonal matrices and $\bm{\Sigma} \in \RR{l \times l}$ is
a diagonal matrix with singular values on its diagonal. The ordering of these
singular values is defined to decrease along the diagonal so that the first SVD
basis vector is more important than subsequent basis vectors, making it easy to
truncate and only use dominant modes in the reduced basis. POD is closely
related to principal component analysis in statistics
\cite{hotelling1933analysis} and Karhunen-Lo\`eve expansion \cite{loeve1955} in
stochastic analysis. Since the objective function in (\ref{eq:pod}) does not
change even though $\bm{\chi}$ is post-multiplied by an arbitrary
$\hat{P}_{i,j}^h \times \hat{P}_{i,j}^h$ orthogonal matrix, the POD procedure
seeks the optimal $\hat{P}_{i,j}^h$-dimensional subspace that captures the
snapshots in the least-squares sense. For more details on POD, we refer to
\cite{hinze2005proper, kunisch2002galerkin}.

Once $\bm{\chi}$ that spans the reduced port space $\hat{P}_{i,j}^h$ is
determined, the rest of the CWROM formulation is the same as the CWFOM
formulation procedure, i.e., finding the bubble function through
\eqref{eq:forcing_bubble}, the interface bubble function through
\eqref{eq:ifunc_bubble}, forming the patched interface basis through
\eqref{eq:ifunc}, assembling the system through Algorithm~\ref{alg:sc_assemble}.
Note that the number of unknowns in the CWROM becomes $\hat{n}_{SC} =
\sum_{\portindex=1}^{n^\gamma} \mathcal{N}_p^\gamma$.

Due to the truncation in the reduced port space as in
\eqref{eq:reduced_port_space}, we can decompose the degrees of
freedom in the component-wise full order model linear system, i.e.,
Eq.~\eqref{eq:cw_linear_system}, into active and inactive ones:
\begin{equation}\label{eq:active_inactive}
  \begin{aligned}
    \cwStiffness(\param) &=
    \bmat{\cwStiffness_{\activeSymbol\activeSymbol}(\param) &
        \cwStiffness_{\activeSymbol\inactiveSymbol}(\param) \\
        \cwStiffness_{\inactiveSymbol\activeSymbol}(\param) &
        \cwStiffness_{\inactiveSymbol\inactiveSymbol}(\param) } \\
    \cwSol(\param) &= \pmat{\cwSol_\activeSymbol(\param) \\
    \cwSol_\inactiveSymbol(\param)} \\
    \cwForce(\param) &= \pmat{\cwForce_\activeSymbol(\param) \\
        \cwForce_\inactiveSymbol(\param)}
  \end{aligned}
\end{equation}
where inactive degrees of freedom $\cwSol_\inactiveSymbol$ correspond to the
coefficients of functions outside the space spanned by $\rifunc{r,k}$.
Setting these degrees of freedom to zero, $\cwSol_\inactiveSymbol = \bm{0}$,
yields the reduced system
\begin{equation}\label{eq:cwrom_linear_system}
  \reduced{\cwStiffness}(\param) \reduced{\cwSol}(\param) = \reduced{\cwForce}(\param)
\end{equation}
where $\reduced{\cwStiffness}(\param) = \cwStiffness_{\activeSymbol\activeSymbol}(\param)$,
$\reduced{\cwSol}(\param) = \cwSol_{\activeSymbol}(\param)$, and
$\reduced{\cwForce}(\param) = \cwForce_{\activeSymbol}(\param)$.
The entries of $\tilde{\cwStiffness}(\param)$ are found just as in Eq. \ref{eq:cwStiffness}, but
using only the members of the reduced skeleton space:
\begin{equation}\label{eq:reduced_cwStiffness}
  \tilde{\cwStiffness}(\param)_{(i,j),(i',j')} = \bilinear{\rifunc{i,j}(\param), \rifunc{i',j'}(\param)},
\end{equation}
and the entries of $\tilde{\cwForce}(\param)$ correspondingly through Eq. \ref{eq:cwForce}.

\subsection{Simplification of the CWFOM for linear parameter dependence}\label{sec:linear_cwrom}
If the bilinear form on each reference component,
$\bilineark{r}{u, v; \paramk{r}}$ is linear in a function of the parameter
$\paramk{r}$, that is:
\begin{equation}\label{eq:linear_dependence}
  \bilineark{r}{u, v; \paramk{r}} = \Theta(\paramk{r}; \paramk{r;0}) \bilineark{r}{u, v; \paramk{r;0}},
\end{equation}
where $\paramk{r;0}$ is some reference value of $\paramk{r}$ and $\Theta: \paramDomaink{r} \to \mathbb{R}^+$
is a function that defines the parameter dependence by scaling the bilinear form,
then the component-wise formulation may be significantly simplified. This simplification is a key
contribution of the present work, as it applies to the topology optimization formulation
developed in the following section. This makes the implementation of our formulation even
more efficient than a component-wise formulation not incorporating the linear simplification.

This simplification eliminates the parameter dependence in the definition of the
interface functions, Eq. \eqref{eq:ifunc}. We do so by defining the lifted port basis functions,
$\psi_{r,j,k} \in \femSpacek{r}$, through the lifting
\begin{equation}\label{eq:elasticity_lifting}
  \begin{split}
    \bilineark{r}{\psi_{r,j,k}, v; \paramk{0}} = 0 \ \forall v \in \bubspace{r;0}{h} \\
    \psi_{r, j, k} = \portbasis{r, j, k} \text{ on }  \hat{\gamma}_{r, j} \\
    \psi_{r, j, k} = 0 \text{ on } \hat{\gamma}_{r, i}, \ i \neq j
  \end{split}
\end{equation}
where the bilinear form $a_r$ is the same one that defines the governing equation in
weak form, as decomposed in Eq. \eqref{eq:decomposed_bilinear}.
Using this definition of the lifted port basis,
and substituting Eq. \eqref{eq:linear_dependence} in Eq. \eqref{eq:ifunc_bubble}, we obtain
\begin{equation}\label{eq:trivial_bubble_function}
  a_{\compIndexMapk{i}}\left(\bub{i,j,k}{h}(\paramk{i}), v; \paramk{i}\right) =
  -a_{\compIndexMapk{i}}\left(\psi_{\compIndexMapk{i},j,k}, v; \paramk{i}\right) =
  \Theta(\paramk{i}) \bilineark{\compIndexMapk{i}}{\psi_{\compIndexMapk{i}, j, k}, v; \paramk{i,0}} =
  0,
\end{equation}
and therefore the bubble function $\bub{i,j,k}{h}$ is parameter independent and equal to zero:
$\bub{i,j,k}{h}(\paramk{i}) = 0$. From Eq. \eqref{eq:ifunc}, this also implies that the
interface basis functions $\phi^h_{\compIndexMapk{i}, j, k}$ are equal to the
lifted port bases $\psi_{\compIndexMapk{i}, j, k}$ and independent of parameter.
A Laplacian lifting as in Eq. \eqref{eq:laplace_lifting} is still possible; in this case,
from Eq. \eqref{eq:ifunc_bubble} one obtains that $\bub{i,j,k}(\paramk{i})$ is equal for
all values of $\paramk{i}$. The lifting in Eq. \eqref{eq:elasticity_lifting} eliminates the
need to solve for the bubble function entirely.

Making $\ifunc{r, j, k}$ parameter independent eliminates the need to solve
Eq. \eqref{eq:ifunc_bubble} for $\bub{i,j,k}{h}$ many times in the online phase
of the component-wise computation, resulting in large reduction in the number
of operations required in the online phase. There is yet another benefit to the linear
simplification, however; because the entries of $\cwStiffness(\param)$ are defined
by applications of the bilinear form (Eq. \eqref{eq:cwStiffness}), they share the
linearity property. That is,
\begin{equation}\label{eq:simplified-cwstiffness}
  \cwStiffness_{(p,k),(p',k')}^i(\paramk{i}) =
    \Theta(\paramk{i})\cwStiffness_{(p,k),(p',k')}^i(\paramk{i;0})
\end{equation}
Therefore, the local Schur complement matrices $\cwStiffness^i(\paramk{i})$ may be
computed for a reference value $\paramk{i;0}$ of the component parameter during the
offline phase, and in the online phase the computation of the local Schur complement
matrices given in Eq. \eqref{eq:local_schur} may be replaced by a simple scaling of
$\cwStiffness^i(\paramk{i;0})$ by $\Theta(\paramk{i})$.

This simplification eliminates most floating point operations in the online phase of
the component-wise computation except for the solution of Eq. \eqref{eq:cw_linear_system}.
Therefore the performance of the algorithm with the linear simplification in effect is
primarily limited only by the cost of assembly and of a linear solver. This simplification
is a key advantage of the component-wise topology optimization formulation that we
demonstrate below.

\section{Component-wise topology optimization}\label{sec:cwtopopt}
We introduce a compliance-based topology optimization formulation based on the component-wise model
developed above, which is particularly useful for designing an optimal lattice-type
structure that can be constructed using a small number of reference components. Our
method is a density-based TO using the solid isotropic material with penalization (SIMP)
method; however, we assign a density parameter to each component, rather than to each
element. This choice of parameterization makes the component bilinear forms linear in the
optimization parameters, and allows the application of the form given in
Eq.~\eqref{eq:simplified-cwstiffness} for the local Schur complement matrices, accelerating
both the forward computation and sensitivity calculations.

\subsection{Forward model}\label{sec:opt_fwdmodel}
The optimization parameter in our formulation is a volume fraction discretized
component-wise, denoted $\param \in [0,1]^{n_I}$. A volume fraction
$\mu_i \in [0,1]$ is assigned to each component, with a value of 0 indicating
that this component is void (omitted from the design) and a value of 1
indicating solid material. For intermediate values of $\mu_i$, the parameter
dependence of the forward model is defined through a SIMP interpolation, defined
below. In practice, we do not let $\paramSymbolk{i}$ to take zero, but a small value,
$\paramSymbolLow > 0$\footnote{we use a value of $10^{-3}$ Pa for $\paramSymbolLow$},
in order to ensure that the resulting problem is well-posed. 

In this work, the forward model is linear elasticity. When expressed in weak form
(\ref{eq:weakform}), the bilinear form is given as
\begin{equation}
  \label{eq:opt_bilin}
    \bilinear{u_h,v;\param} = \int_\sysDomain s(\mu)\elasticityTensor\left[\nabla
    u_h\right] \cdot \nabla v dx
\end{equation}
where $\elasticityTensor$ denotes the symmetric elasticity tensor and $s: \RR{}
\to \RR{}$ denotes a SIMP (Solid Isotropic Material with Penalization) function,
which is defined, for a given exponent $p \in \RR{}_+$, as
\begin{equation}
  s(\mu) \equiv \mu^p + (1 - \mu^p)\frac{E_{min}}{E_0}
\end{equation}
where $E_{min} \in \RR{}_+$ is the minimum Young's modulus, and $E_0 \in \RR{}_+$
is the Young's modulus of fully solid material. In our formulation, $s(\mu)$ is taken
to be piecewise constant, constant over each component.

The linear form is defined by
\begin{equation}
  \label{eq:opt_lin}
  f(v) = \int_\sysDomain f_h \cdot v dx
\end{equation}
where $f_h \in L^2(\sysDomain)$ is the discretized external forcing.

When the bilinear form in Eq. \ref{eq:opt_bilin} is decomposed as in Eq.
\ref{eq:decomposed_bilinear}, it may be written as
\begin{equation}
  \label{eq:decomposed_stiffness}
  a(u_h, v; \param) = \sum_{i=1}^{\numInsta} s(\mu_i)
  \int_{\sysDomaink{i}} \elasticityTensor\left[
  \nabla u_h|_{\sysDomaink{i}}\right] \cdot \nabla v dx
  = \sum_{i=1}^{\numInsta} s(\mu_i) \sbilineark{\compIndexMapk{i}}{u_h|_{\sysDomaink{i}},
  v|_{\sysDomaink{i}}}
\end{equation}
because $s(\mu)$ is constant on $\sysDomaink{i}$,
where $\sbilinearSymbol_{\compIndexMapk{i}}: \femSpacek{i} \times \femSpacek{i} \mapto \RR{}$ is
defined as $\sbilineark{\compIndexMapk{i}}{\solSymbol,\dummyFunc} \equiv \int_{\Omega_i}
\elasticityTensor\left[\nabla \solSymbol\right] \cdot \nabla \dummyFunc dx$ and
is independent of the optimization variable $\param$. Therefore, the simplification
from Section \ref{sec:linear_cwrom} applies and may be used to accelerate the model
evaluations during the optimization iteration.

With the forward model now expressed in the form given in (\ref{eq:weakform}), we may
apply either the CWFOM or the CWROM to solve the model and obtain a solution $u_h(\param)$
or $\reduced{u}_h(\param)$, respectively, along with the objective function to be defined
below, and its sensitivity.

\subsection{Optimization formulation}\label{sec:opt_formulation}
We consider a structural compliance minimization problem on a lattice structure subject
to a volume constraint, which is formulated as:
\begin{equation}
  \label{eq:opt_formulation}
  \begin{split}
    \minimize_{\param \in \RR{\numInsta}} &\quad \compliance(\sol(\param), \param) \\
    \text{subject to} &\quad\ g_0(\param) = \sum_i \mu_i v_i - v_u \leq 0 \\
    & \quad\paramSymbolLow \leq \paramSymbolk{i} \leq 1, \quad i\in\nat{\numInsta}
  \end{split}
\end{equation}
where $\compliance: \femSpace(\sysDomain) \times \RR{\numInsta} \to \RR{}$ will be defined below,
$v_i$ denotes the volume of component $\sysDomain_i$, and $v_u$ denotes the upper bound
for the total volume of the material. Thus the constraint $g_0 \leq 0$ in
(\ref{eq:opt_formulation}) is a limit on the volume (mass) of the system. The state
vector $\sol(\param)$ is found by solving the CWROM for the forward model
presented in Section \ref{sec:opt_fwdmodel}.

The compliance is defined as
\begin{equation}\label{eq:compliance}
  \begin{aligned}
    \compliance\left(u_h(\param); \param\right) &=  \int_\sysDomain
    s(\mu)\elasticityTensor\left[\nabla u_h\right] \cdot \nabla u_h dx    \\
    &=\bilinear{u_h, u_h;\param},
  \end{aligned}
\end{equation}
where the second equality holds by Eq.~\eqref{eq:opt_bilin}. 
Plugging Eq.~\eqref{eq:scrbe_sol} into Eq.~\eqref{eq:compliance}, we obtain the
following equivalent linear algebraic form of compliance:
\begin{equation}\label{eq:discrete-compliance-1}
  \begin{split}
    \dcompliance\left(\cwSol(\param), \dfbubk{1}(\paramSymbolk{1}),\ldots,
    \dfbubk{\numInsta}(\paramSymbolk{\numInsta}); \param\right) = \\
    \cwSol(\param)^T\cwStiffness(\param)\cwSol(\param) + \sum_{i=1}^{\numInsta}
    \lineark{\compIndexMapk{i}}{\bub{i}{f;h}(\paramSymbolk{i}); \paramSymbolk{i}} +
    \sum_{i=1}^{\numInsta} \sum_{p=1}^{n^\gamma}\sum_{k=1}^{\spaceDimension{p}{\gamma}}
    \cwSolk{p,k} \bilineark{i}{\bub{i}{f}(\mu_i), \ifunc{p,k}; \mu_i}
  \end{split}
\end{equation}
\begin{equation}\label{eq:discrete-compliance}
  = \cwForce^T\cwSol(\param) + \sum_{i=1}^{\numInsta} \femForcek{i}^T
    \dfbubk{i}(\paramSymbolk{i})
\end{equation}
where $\cwStiffness$, $\cwForce$ and $\cwSol$ are defined in
Eqs.~\eqref{eq:cw_linear_system} - \eqref{eq:cwForce}. $\femForcek{i}$ and
$\dfbubk{i}$ are component-level quantities: $\dfbubk{i}(\paramSymbolk{i}) \in
\RR{\spaceDimension{\compIndexMapk{i}}{h}}$ is a vector of coefficients of
$\bub{i}{f;h}$ in the basis of the bubble space
$\bubspace{\compIndexMapk{i};0}{h}$, and $\femForcek{i} \in
\RR{\spaceDimension{\compIndexMapk{i}}{h}}$ is the vector discretizing the
linear form such that $f(v) = \femForcek{i}^T \boldsymbol{v}$, $\forall v \in
\bubspace{\compIndexMapk{i};0}{h}$, with $\boldsymbol{v}$ the vector of
coefficients of $v$ in the bubble space basis, as for $\dfbubk{i}$. Here we have
assumed that the forcing is independent of the parameter $\param$, as it is in
the numerical examples presented below. Note that $\cwForce$ for our particular
choice of optimization variable, i.e., volume fraction, is independent of
$\param$ as well because the bilinear form depends linearly on $s(\param)$. The
linear form appears in the second term in Eq.~\eqref{eq:discrete-compliance-1}
due to Eq.~\eqref{eq:forcing_bubble}. The last term in
Eq.~\eqref{eq:discrete-compliance-1} vanishes because of our use of the
elasticity lifting, Eq.~\eqref{eq:elasticity_lifting}, which means that the
bilinear form applied to an interface function and any function in the bubble
space is identically zero.

Analogously to the decomposition of the bilinear form in Eq.~\eqref{eq:decomposed_bilinear},
the compliance may also be decomposed as
\begin{equation}\label{eq:decomposed_compliance}
  \compliance\left(\sol(\param); \param\right) = \sum_{i=1}^{\numInsta}
  \compliance_{i}\left(\sol(\paramSymbolk{i})|_{\sysDomaink{i}};
  \paramSymbolk{i}\right)
\end{equation}
where each component compliance $\compliancek{i}: X^h_{\compIndexMapk{i}} \times
\RR{} \to \RR{}$ is defined as
\begin{equation}\label{eq:comp_compliance_def}
  \begin{aligned}
    \compliancek{i}\left(\sol(\paramSymbolk{i})|_{\sysDomaink{i}};
    \paramSymbolk{i}\right) &\equiv \int_{\sysDomaink{i}} s(\paramSymbolk{i})
    \elasticityTensor\left[ \nabla
    \sol(\paramSymbolk{i})|_{\sysDomaink{i}}\right] \cdot \nabla
    \sol(\paramSymbolk{i})|_{\sysDomaink{i}}\ dx \\
    &= \bilineark{i}{\sol(\paramSymbolk{i})|_{\sysDomaink{i}},
    \sol(\paramSymbolk{i})|_{\sysDomaink{i}}; \paramSymbolk{i}} \\
    &= s(\paramSymbolk{i})
    \sbilineark{i}{\sol(\paramSymbolk{i})|_{\sysDomaink{i}},
    \sol(\paramSymbolk{i})|_{\sysDomaink{i}}}
  \end{aligned}
\end{equation}

From the form of $\reduced{u}_h(\param)$ given in \eqref{eq:reduced_sol} and
the definition of $\reduced{\cwStiffness}$ in Eq.~\eqref{eq:reduced_cwStiffness},
we note that the component compliance can be equivalently written in the following
linear algebraic form as $\dcompliancek{i}: \RR{\mathcal{N}_i} \times \RR{} \to \RR{}$:
\begin{equation}\label{eq:comp_compliance_la}
  \begin{aligned}
    \dcompliancek{i}\left(\cwSolk{i}(\param), \dfbubk{i}(\paramSymbolk{i});
    \param\right)
    &= \cwSolk{i}(\param)^T \cwStiffness^i(\paramSymbolk{i})
    \cwSol_i(\param) + \femForcek{i}^T \dfbubk{i}(\paramSymbolk{i}) \\
    &= s(\paramSymbolk{i}) \cwSolk{i}(\param)^T \scwStiffness^i
    \cwSol_i(\param) + \femForcek{i}^T \dfbubk{i}(\paramSymbolk{i})
  \end{aligned}
\end{equation}
where $\mathcal{N}_i \equiv \sum_{j=1}^{n_{\compIndexMapk{i}}^\gamma}
\mathcal{N}_{\compIndexMapk{i},j}^\gamma$ denotes the total number of degrees of
freedom in all the ports of the $i$th instantiated component and $\cwSol_i \in
\RR{\mathcal{N}_i}$ denotes the coefficient vector in the $i$th component whose
entries consist of $\cwSol_{\globpmapk{i}(j),k}, j \in
\nat{n_{\compIndexMapk{i}}^\gamma}, k \in
\nat{\mathcal{N}_{\compIndexMapk{i},j}^\gamma}$, and $\cwStiffness^i \in
\RR{\mathcal{N}_i \times \mathcal{N}_i}$ denotes the $i$th stiffness matrix
whose entries consist of $\cwStiffness_{(j,k),(j',k')}^i$ defined in
(\ref{eq:local_schur}). Because of the linear dependence of the bilinear form
on $s(\mu)$, we make use of the development in Sec.~\ref{sec:linear_cwrom} and define
a parameter-independent component stiffness matrix:
\begin{equation}\label{eq:independent_cwStiffness}
  \scwStiffness_{(j,k),(j',k')}^{i} =
  \sbilineark{\compIndexMapk{i}}{\phi_{i,j,k}^h, \phi_{i,j',k'}^h}
\end{equation}
where the interface functions do not depend on $\mu_i$; this parameter-independent
stiffness appears in the second equivalence in equation \eqref{eq:comp_compliance_la}.

We also require the sensitivity of the compliance objective to $\param$ for use in a
gradient-based optimization. To derive the sensitivity, we begin by defining the following
residuals:
\begin{equation}\label{eq:cwresidual}
  \begin{aligned}
    \zerovec = \cwRes(\cwSol(\param);\param) &\equiv \cwForce -
    \cwStiffness(\param)\cwSol(\param)
  \end{aligned}
\end{equation}
\begin{equation}\label{eq:forcebubbleresidual}
  \begin{aligned}
    \zerovec = \bubres{i}(\dfbubk{i}(\paramSymbolk{i}); \paramSymbolk{i}) &\equiv
    \femStiffnessk{i}(\paramSymbolk{i})\dfbubk{i}(\paramSymbolk{i}) - \femForcek{i}
  \end{aligned},
\end{equation}
where $\femStiffnessk{i}(\paramSymbolk{i}) \in \RR{\spaceDimension{\compIndexMapk{i}}{h}
\times\spaceDimension{\compIndexMapk{i}}{h}}$ is the finite element stiffness matrix
discretizing $a_i(u,v;\paramSymbolk{i})$ for $u, v \in \bubspace{\compIndexMapk{i};0}{h}$.

The sensitivity of compliance is given by
\begin{equation}\label{eq:comp-sensitivity} 
  \begin{aligned}
    \frac{d\dcompliance}{d\param} &=
    \frac{\partial\dcompliance}{\partial\param} +
    \frac{\partial\dcompliance}{\partial\cwSol}
    \frac{d\cwSol}{d\param} +
    \sum_{i=1}^{\numInsta} \frac{\partial\dcompliance}{\partial\dfbubk{i}}
    \frac{d\dfbubk{i}}{d\paramSymbolk{i}}\canonicalk{i}^T
  \end{aligned}
\end{equation}
(note that $\dfbubk{i}$ depends only on $\paramSymbolk{i}$).
The expression for $\frac{d\cwSol}{d\param}$ can be obtained from the
derivative of the residual \eqref{eq:cwresidual}:
\begin{equation}\label{eq:res-sensitivity}
  \begin{aligned}
    \zerovec = \frac{d\cwRes}{d\param} &=
    \frac{\partial\cwRes}{\partial\param} +
    \frac{\partial\cwRes}{\partial\cwSol}
    \frac{d\cwSol}{d\param},
  \end{aligned}
\end{equation}
from which follows
\begin{equation}\label{eq:direct-sensitivity}
  \frac{d\cwSol}{d\param} = -\left (\frac{\partial\cwRes}{\partial\cwSol}
  \right)^{-1} \frac{\partial\cwRes}{\partial\param}.
\end{equation}
Plugging Eq.~\eqref{eq:direct-sensitivity} into the second term in 
Eq.~\eqref{eq:comp-sensitivity}, we obtain
\begin{equation}\label{eq:comp-sensitivity2} 
  \begin{aligned}
    \frac{\partial\dcompliance}{\partial\cwSol}
    \frac{d\cwSol}{d\param} = -2\cwStiffness(\param)\cwSol(\param)
    \left(\frac{\partial\cwRes}{\partial\cwSol}\right)^{-1} \frac{\partial\cwRes}{\partial\param}
    = -\lagrangeMultiplier^T \frac{\partial\cwRes}{\partial\param},
  \end{aligned}
\end{equation}
where the Lagrange multiplier, $\lagrangeMultiplier\in\RR{}$, can be obtained by
solving the following adjoint problem:
\begin{equation}\label{eq:adjoint-problem}
   \left( \frac{\partial\cwRes}{\partial\cwSol} \right)^T\lagrangeMultiplier =
   \cwStiffness(\param)\lagrangeMultiplier =
   \frac{\partial\dcompliance}{\partial\cwSol} =
   2\cwStiffness(\param)\cwSol(\param).
\end{equation}
Using the definition of the residual,
Eq.~\eqref{eq:cwresidual}, and noting that $\lagrangeMultiplier = 2\cwSol$, the
sensitivity becomes:
\begin{equation}\label{eq:comp-sensitivity3} 
  \begin{aligned}
    \frac{d\dcompliance}{d\param} &= \frac{\partial\dcompliance}{\partial\param} -2\cwSol(\param)^T 
    \frac{d\cwStiffness}{d\param}\cwSol(\param) +
    \sum_{i=1}^{\numInsta} \frac{\partial\dcompliance}{\partial\dfbubk{i}}
    \frac{d\dfbubk{i}}{d\paramSymbolk{i}}\canonicalk{i}^T
  \end{aligned}
\end{equation}

The second term in Eq.~\eqref{eq:comp-sensitivity3} may be simplified similarly
using the residual defined in Eq.~\eqref{eq:forcebubbleresidual}:
\begin{equation}\label{eq:bub-sensitivity1}
  \zerovec = \frac{d\bubres{i}}{d\paramSymbolk{i}} =
  \frac{\partial\bubres{i}}{\partial\paramSymbolk{i}} +
  \frac{\partial\bubres{i}}{\partial\dfbubk{i}}\frac{d\dfbubk{i}}{d\paramSymbolk{i}}
\end{equation}
\begin{equation}\label{eq:bub-sensitivity2}
  \frac{d\dfbubk{i}}{d\paramSymbolk{i}} =
  -\left(\frac{\partial\bubres{i}}{\partial\dfbubk{i}}\right)^{-1}
  \frac{\partial\bubres{i}}{\partial\paramSymbolk{i}}
\end{equation}
From Eq.~\eqref{eq:bub-sensitivity2} we obtain the following sensitivity of compliance
to the forcing bubble functions:
\begin{equation}\label{eq:bub-sensitivity3}
  \frac{\partial\dcompliance}{\partial\dfbubk{i}}\frac{d\dfbubk{i}}{d\param} =
  -\bubLagrangeMultiplier{i}^T \frac{\partial\bubres{i}}{\partial\paramSymbolk{i}}
  \canonicalk{i}^T
\end{equation}
with the Lagrange multipliers $\bubLagrangeMultiplier{i}$ given through the solution of the
adjoint problems
\begin{equation}
  \frac{\partial\bubres{i}}{\partial\dfbubk{i}} \bubLagrangeMultiplier{i} =
  \frac{\partial\dcompliance}{\partial\dfbubk{i}},
\end{equation}
yielding $\bubLagrangeMultiplier{i} = 2\dfbubk{i}$; this in turn implies
\begin{equation}
  \frac{\partial\dcompliance}{\partial\dfbubk{i}}\frac{d\dfbubk{i}}{d\param} =
  -2\dfbubk{i}(\paramSymbolk{i})^T \frac{\partial\femStiffnessk{i}}{\partial\paramSymbolk{i}}
  \dfbubk{i}(\paramSymbolk{i}) \canonicalk{i}^T,
\end{equation}
finally leading to the following form for the sensitivity of compliance:
\begin{equation}\label{eq:comp-sens-final}
  \frac{d\dcompliance}{d\param} = -\cwSol(\param)^T \frac{d\cwStiffness}{d\param}
  \cwSol(\param) - 
  \sum_{i=1}^{\numInsta} \dfbubk{i}(\paramSymbolk{i})^T \frac{\partial\femStiffnessk{i}}{\partial\paramSymbolk{i}}
  \dfbubk{i}(\paramSymbolk{i}) \canonicalk{i}^T
\end{equation}

However, because of the structure of $\cwStiffness(\param)$, the dependence of
$\dcompliance$ on $\paramSymbolk{i}$ may be expressed purely in terms of quantities
defined for instantiated component $i$:
\begin{equation}
  \begin{aligned}
    \frac{d\dcompliance}{d\paramSymbolk{i}} &= -\frac{ds}{d\mu_i}
    \left(\cwSol_i(\param)^T \scwStiffness^i \cwSol_i(\param) +
      \dfbubk{i}(\paramSymbolk{i})^T \sfemStiffnessk{i}\dfbubk{i}(\paramSymbolk{i})
    \right)
  \end{aligned}
\end{equation}
where we have again made use of the linearity of $a_i$ with respect to
$\paramSymbolk{i}$ to define a parameter-independent component stiffness matrix
$\sfemStiffnessk{i} = \femStiffnessk{i}(\boldsymbol{1})$. Thanks to this component-wise
decomposition, the computation of compliance is accelerated and may be parallelized in
the same fashion as the assembly of the static condensation system to solve the forward
problem. We note that the terms including $\dfbubk{i}(\paramSymbolk{i})$ are only
non-zero when there is a forcing applied on the part of a component's domain not
including port domains. In many cases, including the numerical examples presented here,
this contribution disappears for all but a few components; in our examples, forcing is
only applied on port domains, eliminating these terms entirely.

The sensitivity computation of the volume constraint $g_0(\param)$ in
(\ref{eq:opt_formulation}) is straightforward and omitted here. These
sensitivities, along with the forward model evaluation, can be used in any
gradient-based optimization solver, such as the method of moving asymptotes
(MMA) \cite{svanberg1987mma}, the interior-point method
\cite{forsgrenGill1998SIAM, wachter2006opt, petra2018newton}, and the sequential
quadratic programming method \cite{boggs1995seq,gill2005snopt} to solve the
optimization problem (\ref{eq:opt_formulation}).

\subsubsection{Optimization in the CWROM context}
The discussion above applies to the CWROM case as well, with the ROM quantities
substituted for their full-order versions. To distinguish full order and
reduced order quantities, we define a reduced order compliance objective:
\begin{equation}\label{eq:reduced_compliance}
  \rcompliance(\rsol(\param); \param) = a(\rsol, \rsol; \param)
\end{equation}
which may be equivalently written in linear algebraic form as
\begin{equation}\label{eq:rcompliance_la}
  \rcompliance(\param) = \reduced{\cwForce}^T \reduced{\cwSol}(\param)
\end{equation}
Following the derivation above, we also have that the sensitivity of the reduced
compliance is given by
\begin{equation}\label{eq:rcompliance_sens}
  \frac{d\rcompliance}{d\mu_i} = -\frac{ds}{d\mu_i}\left(
  \reduced{\cwSol}_i(\mu_i)^T \hat{\cwStiffness}^i \reduced{\cwSol}_i(\mu_i)
  + \dfbubk{i}(\paramSymbolk{i})^T \sfemStiffnessk{i}\dfbubk{i}(\paramSymbolk{i})\right)
\end{equation}
where the parameter-independent reduced component stiffness matrix is given by
\begin{equation}
  \hat{\cwStiffness}^i_{(j,k),(j',k')} =
  \sbilineark{\compIndexMapk{i}}{\reduced{\phi^h}_{i,j,k}, \reduced{\phi^h}_{i, j',k'}}
\end{equation}

\section{Error bounds}\label{sec:error}
We are interested in deriving the various error bounds. First, the energy norm,
$\energyNorm{\cdot}: \HH(\sysDomain)\mapto\RR{}$, is defined as
$\energyNorm{\dummyFunc} \equiv \sqrt{\bilinear{\dummyFunc,\dummyFunc;\param}}$,
$\forall \dummyFunc\in\HH(\sysDomain)$.  Then we define the following error
quantities: the solution error, $\solerr: \paramDomain\mapto\HH(\sysDomain)$, is
defined as $\solerr(\param) \equiv \sol(\param) - \reducedSolution(\param)$.
The compliance error, $\compErr: \paramDomain\mapto\RR{}$, is defined as
$\compErr(\param) \equiv \compliance\left(\sol(\param); \param\right) -
\acompliance\left(\reducedSolution(\param); \param\right)$.  Finally, the
component compliance sensitivity error, $\compSensErr:\paramDomain\mapto\RR{}$,
is defined as $\compSensErr(\param) \equiv \frac{d\dcompliance}{d\param}(\cwSol)
- \frac{d\dacompliance}{d\param}(\reducedcwSol)$.

The following error bounds will be derived:
\begin{itemize}
  \item Solution error: $\energyNorm{\solerr(\param)} \leq
    \frac{c_2}{\sigma_{\cwStiffness,\text{min}}}\ltwoNorm{\cwRes} $ 
  \item Compliance error: $\absNorm{\compErr(\param)} \leq
    \frac{c_2\ltwoNorm{\cwForce}}{c_1\sigma_{\cwStiffness,\text{min}}}
    \ltwoNorm{\cwRes} $ 
  \item Compliance sensitivity error: $\ltwoNorm{\compSensErr(\param)}\leq
    \frac{\stiffnessSensitivityNorm}{\sigma_{\cwStiffness,\text{min}}^2}
    \ltwoNorm{\cwRes}^2 +
    2\frac{\stiffnessSensitivityNorm\normreducedCWsol}{\sigma_{\cwStiffness,\text{min}}}
    \ltwoNorm{\cwRes}$,
\end{itemize}
where each constant in front of the residual norm will be defined later when
each theorem is stated.

Before stating the derivation of these bounds, we note that all terms involving
the forcing bubble functions $\bub{i}{f}$ vanish; this is because the bubble functions
are not approximated in our method (although they are in the original SCRBE method),
and thus each term containing only bubble functions cancels when subtracted from the
FOM quantity. The component-wise form of the displacement field, Eq. \eqref{eq:extended_reduced_sol},
the compliance, and the compliance sensitivity all have terms containing only bubble
functions or only reduced order quantities; therefore, bubble function terms appear
nowhere in the derivation of these bounds.

Before diving into the derivation of the error bounds above, we first note that
the reduced solution in \eqref{eq:reduced_sol} can be re-written as the
following extended form:
\begin{equation}\label{eq:extended_reduced_sol}
  \reducedSolution(\param) = \sum_{i=1}^{\numInsta} \bub{i}{f;h}(\paramk{i}) +
  \sum_{p=1}^{n^\gamma} \sum_{k=1}^{\mathcal{N}_p^\gamma}
  \extendedCWsolk{p,k}(\param)\interfacek{p,k}(\param),
\end{equation}
where the value of $\extendedCWsolk{p,k}$ is determined by 
\begin{equation}\label{eq:extendedCWsol}
  \begin{split}
    \extendedCWsolk{p,k} = \reducedCWsolk{p,k} &\quad (p,k) \in \portReducedSet
    \\
    \extendedCWsolk{p,k} = 0 &\quad (p,k) \notin \portReducedSet,
  \end{split}
\end{equation}
where $\portReducedSet$ is the set of the port and its degree of freedom pairs
that are selected in the port reduction. Subtracting
\eqref{eq:extended_reduced_sol} from \eqref{eq:scrbe_sol}, the solution error is
expressed as 
\begin{equation}
  \label{eq:sol_error}
  \solerr(\param) = \sum_{p=1}^{n^\gamma} \sum_{k=1}^{\mathcal{N}_p^\gamma}
  (\CWsolk{p,k}(\param)-\extendedCWsolk{p,k})(\param)\interfacek{p,k}(\param).
\end{equation}
Note that $\solerr(\param)\in\skeletonSpace$, where $\skeletonSpace$ is defined
in \eqref{eq:skeleton}. Based on the definition of the component-wise stiffness
matrix in Eq.~\eqref{eq:cwStiffness}, the energy norm of the solution error is
the same as the $\cwStiffness$-induced norm of the error in component-wise
coefficient error, i.e., 
\begin{equation}\label{eq:equivalenceNorm}
  \energyNorm{\solerr(\param)} = \KinducedNorm{\cwsolerr(\param)}, 
\end{equation}
where the component-wise coefficient error,
$\cwsolerr:\paramDomain\mapto\RR{n_{SC}}$, is defined as $\cwsolerr(\param)
\equiv \cwSol(\param) - \extendedCWsol(\param)$ and the $\cwStiffness$-induced
norm, $\KinducedNorm{\cdot}:\RR{n_{SC}}\mapto\RR{}$, is defined as
$\KinducedNorm{\dummyVec} \equiv \sqrt{\dummyVec^T\cwStiffness\dummyVec}$,
$\forall \dummyVec\in\RR{n_{SC}}$.
In vector form, the extended component-wise solution, $\extendedCWsol(\param)$,
is nothing more than 
\begin{equation}\label{eq:extended_cwsol}
  \extendedCWsol(\param) = \pmat{\cwSol_{\activeSymbol} \\ \zerovec},
\end{equation}
where $\cwSol_{\activeSymbol}$ can be obtained by solving
Eq.~\eqref{eq:cwrom_linear_system}.  Now, the extended component-wise solution,
$\extendedCWsol(\param)$, will make the residual non-zero, so we define the
corresponding residual, $\cwRes:\paramDomain\mapto\RR{n_{SC}}$ as
\begin{equation}\label{eq:extendedcwresidual}
  \begin{aligned}
    \cwRes(\extendedCWsol(\param);\param) &= \cwForce(\param) -
    \cwStiffness(\param)\extendedCWsol(\param) \\
    &= \cwStiffness(\param)\cwsolerr(\param),
  \end{aligned}
\end{equation}
where the equality of the second line above is due to
Eq.~\eqref{eq:cw_linear_system}.

\begin{thm}\label{thm:sol-errorbound}
  {\bf A posteriori residual-based error bound for solution state} 
  Let $c_2>0$
  the norm equivalence constant, i.e., $\KinducedNorm{\cdot} \leq c_2
  \ltwoNorm{\cdot}$, and the minimum singular value of $\cwStiffness$ is denoted
  as $\sigma_{\cwStiffness,\text{min}}>0$, then for any given $\param \in
  \RR{\numInsta}$, the following a posteriori error bound holds:
  \begin{equation}\label{eq:solerrbound}
    \energyNorm{\solerr(\param)} \leq
    \frac{c_2}{\sigma_{\cwStiffness,\text{min}}} \ltwoNorm{\cwRes} 
  \end{equation}
\end{thm}
\begin{proof}
  By the h\"{o}lder's inequality and Eq.~\eqref{eq:extendedcwresidual}, we have
  \begin{equation}\label{eq:holder}
    \ltwoNorm{\cwsolerr(\param)} \leq \ltwoNorm{\cwStiffness^{-1}}
    \ltwoNorm{\cwRes}.
  \end{equation}
  Then, the desired error bound follows due to the norm equivalence relation and
  Eq.~\eqref{eq:equivalenceNorm}.
\end{proof}
\begin{thm}\label{thm:compliance-errorbound}
  {\bf A posteriori residual-based error bound for compliance objective
  function} Let $c_1$, $c_2>0$ the norm equivalence constants, i.e., $c_1
  \ltwoNorm{\cdot} \leq \KinducedNorm{\cdot} \leq c_2 \ltwoNorm{\cdot}$, and the
  minimum singular value of $\cwStiffness$ is denoted as
  $\sigma_{\cwStiffness,\text{min}}>0$, then for any given $\param \in
  \RR{\numInsta}$, the following a posteriori error bound for the compliance
  objective function holds:
  \begin{equation}\label{eq:solerrbound}
    \absNorm{\compErr(\param)} \leq
    \frac{c_2\ltwoNorm{\cwForce}}{c_1\sigma_{\cwStiffness,\text{min}}}
    \ltwoNorm{\cwRes} 
  \end{equation}
\end{thm}
\begin{proof}
  By Eq.~\eqref{eq:discrete-compliance}, the compliance error can be written as
  \begin{equation}\label{eq:discrete_compliance_err}
    \begin{aligned}
      \compErr(\param) &= \cwForce^T\cwSol -
      \cwForce_{\activeSymbol}^T\cwSol_{\activeSymbol}\\
        &=\cwForce^T(\cwSol - \extendedCWsol),
    \end{aligned}
  \end{equation}
  where the second equality comes from the fact that
  $\cwForce_{\activeSymbol}^T\cwSol_{\activeSymbol} = \cwForce^T\extendedCWsol$
  due to the definition of $\extendedCWsol$ in Eq.~\eqref{eq:extended_cwsol}.
  By the h\"{o}lder's inequality, we have
  \begin{equation}\label{eq:holder}
    \begin{aligned}
      \absNorm{ \compErr(\param) } &\leq \ltwoNorm{\cwForce}\ltwoNorm{\cwSol -
        \extendedCWsol} \\
        &\leq \frac{\ltwoNorm{\cwForce}}{c_1}\KinducedNorm{\cwSol -
        \extendedCWsol},
    \end{aligned}
  \end{equation}
  where the second inequality above comes from the equivalence relation of the
  norms. 
  Then, the desired error bound follows by Theorem~\ref{thm:sol-errorbound}.
\end{proof}
\begin{thm}\label{thm:compliancesensitivity-errorbound}
  {\bf A posteriori residual-based error bound for compliance sensitivity} 
  The minimum singular value of $\cwStiffness$ is denoted as
  $\sigma_{\cwStiffness,\text{min}}>0$ and let $\stiffnessSensitivityNorm =
  \ltwoNorm{\frac{d\cwStiffness}{d\param}}$ and $\normreducedCWsol =
  \ltwoNorm{\reducedCWsol}$.  Then for any given $\param \in \RR{\numInsta}$,
  the following a posteriori error bound holds:
  \begin{equation}\label{eq:solerrbound}
    \ltwoNorm{\compSensErr(\param)} \leq
    \frac{\stiffnessSensitivityNorm}{\sigma_{\cwStiffness,\text{min}}^2}
    \ltwoNorm{\cwRes}^2 +
    2\frac{\stiffnessSensitivityNorm\normreducedCWsol}{\sigma_{\cwStiffness,\text{min}}}
    \ltwoNorm{\cwRes}
  \end{equation}
\end{thm}
\begin{proof}
  By Eq.~\eqref{eq:comp-sens-final}, the compliance sensitivity error can be
  written as
  \begin{equation}\label{eq:discrete_compliance-sensitivity-error}
    \begin{aligned}
      \compSensErr(\param) &= \cwSol_{\activeSymbol}^T
        \frac{d\reducedcwStiffness}{d\param} \cwSol_{\activeSymbol} -
        \cwSol^T \frac{d\cwStiffness}{d\param} \cwSol \\
      &= \extendedCWsol^T \frac{d\cwStiffness}{d\param} \extendedCWsol -
      \cwSol^T \frac{d\cwStiffness}{d\param} \cwSol,
    \end{aligned}
  \end{equation}
  where the second equality comes from the fact that $\cwSol_{\activeSymbol}^T
  \frac{d\reducedcwStiffness}{d\param} \cwSol_{\activeSymbol} =
  \extendedCWsol^T \frac{d\cwStiffness}{d\param} \extendedCWsol$ 
  due to the definition of $\extendedCWsol$ in Eq.~\eqref{eq:extended_cwsol}.
  Note that the following identity holds:
  \begin{equation}\label{eq:identity-sensitivity-error}
    \extendedCWsol^T \frac{d\cwStiffness}{d\param} \extendedCWsol -        
    \cwSol^T \frac{d\cwStiffness}{d\param} \cwSol = - (\cwSol -
    \extendedCWsol)^T\frac{d\cwStiffness}{d\param} (\cwSol - \extendedCWsol) -
    2\extendedCWsol^T\frac{d\cwStiffness}{d\param} (\cwSol - \extendedCWsol).
  \end{equation}
  Applying the H\"older's inequality, the equivalence norm
  \eqref{eq:equivalenceNorm} and \eqref{eq:holder} to
  Eq.~\eqref{eq:identity-sensitivity-error}, the desired error bound follows. 
\end{proof}
  
\section{Computational Costs}\label{sec:computational-costs}
A key advantage of the CWROM methodology is the reusability of components; given
a set of trained components, we may now solve any system composed of connected
instantiations of those reference components. Furthermore, because the forward model
only needs to be solved over the domain of two instantiated components for the training
procedure, the training of the CWROM is much more economical than the training of a
conventional ROM that takes snapshots of the entire system state, since the problem
solved in the offline phase is of a much smaller dimension than that solved in the
online phase. In this section we quantify the training and solution costs in order to
predict the speedup that can be expected from use of the CWROM.

Here we take \(\mathcal{N}^\gamma = \spaceDimension{p}{\gamma}\) to be the same for all
ports, and \(\tilde{\mathcal{N}}^\gamma\) be its reduced counterpart. We let \(N_{p}\) be the
number of ports per component for all components in a system. Finally, the full dimension
of the finite element function space over a component's mesh is denoted by \(N_{h} = \left|X_h(\Omega_i)\right|\),
equal for each \(\Omega_i\).

\subsection{Offline costs}
As discussed above, the training cost for the CWROM scales with the dimension of the
discretized components, not with the size of an assembled system from those components.
The contributing costs in the offline phase are:

\begin{itemize}
    \item The solution process of Sturm-Liouville eigenproblem for each port,
        Eq. \eqref{eq:sturm-liouville}, to obtain the port basis used in training.
    \item The solution process of Step 4 of Algorithm~\ref{alg:pairwise_training};
      this is the dominant cost involved in training.
    \item The POD process \eqref{eq:pod}
    \item Lifting of the computed basis to form the reduced skeleton space
\end{itemize}

The eigensolve cost is dominated by the actual eigendecomposition, requiring
\(O\left((\mathcal{N}^\gamma)^3\right)\) floating point operations (FLOPs).
Assuming POD is performed by computing the SVD of a snapshot matrix, the work
required in this step is \(O\left(\mathcal{N}^\gamma N_{samples}^2\right)\).
Generally, both of these costs are dominated by the cost of the solution process for
the two-component system. If this solution is computed using the conventional FEM
technique, the largest floating point cost is the cost of solving the assembled linear
system. While the complexity of this solve is difficult to predict, it will typically
be on the order of \(N_h^2\); since \(N_h\) is typically at least an order of magnitude
larger than \(\mathcal{N}^\gamma\) and the solve must be repeated \(N_{samples}\) times,
this cost will significantly outweigh the cost of the eigensolve and the POD procedure.
The lifting of the computed basis is also significantly more expensive than these two
steps, since it requires \(\tilde{\mathcal{N}}^\gamma\) FEM solves over the domain of
a single component. This cost is of the same order as the cost of training; however,
in general \(\tilde{\mathcal{N}}^\gamma\) is significantly less than \(N_{samples}\)
so that the training procedure is the dominant cost.

We actually choose to solve the two-component system using the CWFOM. In this
case, the cost of the solve scales as discussed in the next sub-section on the cost of
the online phase; however, there is a preliminary computational cost to build the component-wise
model. This cost consists of an eigensolve to compute a basis for the ports of each
component, and the lifting of that basis to form a skeleton space for each component.
The latter cost is again the most significant, since it requires \(\mathcal{N}^\gamma\)
solutions of the governing equation over a component at a cost roughly proportional to
\(N_h^2\). Once this basis is built, however, the solution on the two component system
is quite efficient, so that the cost of building the CWFOM may be amortized over the
collection of many snapshots.

\subsection{Online cost}
Once the component-wise training is finished, the online phase consists of
\begin{itemize}
  \item Assembly (Algorithm~\ref{alg:sc_assemble}) and the linear system solve. Assembly requires
    computation of the bubble functions $b_i^{f;h}(\paramk{i})$ and
    $b^h_{i,j,k}(\paramk{i})$, and computation of the local Schur complement matrices
    $\cwStiffness^i(\bm{\mu}_i)$; these costs outweigh the cost of assembling
    the condensed forcing vector $\cwForce$, which is omitted here.
  \item Solution of the Schur complement system (Eq. \eqref{eq:cw_linear_system})
  \item Reconstructing the solution field and computing output quantities
\end{itemize}

For each component, the computation of the bubble functions requires
$n_{\compIndexMapk{i}}^\gamma \times \tilde{\mathcal{N}}^\gamma + 1$ solutions to a FEM problem
on the component's discretization. In practice for parameter dependent problems,
an affine decomposition of the bilinear form will be stored so that the computational
cost of this step is primarily due to the linear solve phase of the computation,
not assembly. We again take the cost of this solve to be $O(N_h^2)$, with the caveat
that time complexity of sparse linear solvers can be difficult to predict. Therefore
the cost of computing bubble functions is
$O\left(n_{\compIndexMapk{i}}^\gamma \times \tilde{\mathcal{N}}^\gamma \times N_h^2
+ N_h^2\right)$. This cost may be significantly reduced in implementations by storing
an affine decomposition of the factorization of the bilinear form matrix as well as
the matrix itself, so that only the application of the factorization is required in the
online phase.

The computation of the local Schur complement matrix consists of a series of
applications of the bilinear form, $a_{\compIndexMapk{i}}\left(\phi_{i,j,k}^h(\paramk{i}),
\phi_{i,j',k'}^h(\paramk{i});\paramk{i}\right)$. Assuming that an affine decomposition
of the matrix representation of the bilinear form is stored and that it is sparse,
each application of the bilinear form requires only $O(N_h)$ operations. Thus the
computation of the local Schur complement requires
$O\left(n_{\compIndexMapk{i}}^\gamma \times \tilde{\mathcal{N}}^\gamma \times N_h\right)$
operations, dominated by the cost of computing the bubble functions required to form the
skeleton space.

Finally, we solve the assembled linear system of size $n_{SC} \times n_{SC}$ and obtain
the solution field and output quantities. The system
has block sparse structure, but in the case where the number of components is small it is
actually quite dense so that algorithms for sparse systems are not beneficial. In this
case, then, we take the complexity of the solution to be $O(n_{SC}^3)$. In large systems,
such as the lattice systems where we will apply the CWROM, sparse solvers may be used and
decrease complexity to $O(n_{SC}^2)$. Reconstructing the solution consists of scaling of
the patched interface basis functions $\ifunc{p,k}(\bm{\mu})$ by the coefficients
$\cwSol_{p,k}(\bm{\mu})$ and summation to form the solution, at a cost of
$O(\tilde{\mathcal{N}}^\gamma N_h)$ per component. In our use case, the output quantities
of interest are the compliance and its sensitivity. The cost of computing the compliance
is also $O(N_h)$ since it is a simple dot product; the cost of computing its sensitivity
is, as well, since it consists of a dot product and a sparse matrix-vector product, but
the constant in the sensitivity case is larger.

We note that all of the steps above may be trivially parallelized by distributing operations
on a per-component basis except for the linear system solve. Therefore, we expect the
cost of the linear system solve, $O(n_{SC}^2)$, to be the dominant factor in a well-optimized
parallel implementation.

\subsubsection{Simplification for the linear case}
In the case that the simplification in Sec. \ref{sec:linear_cwrom} is valid, i.e.,
the bilinear form is linear in a function of $\bm{\mu}$, the local Schur complement
contribution $\cwStiffness^i(\bm{\mu}_i)$ may be computed for a reference value of
$\bm{\mu}$ and simply scaled during assembly in the online phase. This reduces the cost
of computing the local Schur complement during assembly, but more importantly, when
using the linear simplification the patched interface basis functions
$\ifunc{p,k}$ are parameter independent, completely eliminating the cost of computing
the bubble functions related to the interface functions. A single bubble function solve
for $\bub{i}{f;h}(\bm{\mu}_i)$ is still required if the body forcing on the component is
non-zero, but the simplification still mostly removes the $O(N_h^2)$ cost of computing
bubble functions. In this work, we allow forces to be applied only over ports so that
no bubble functions at all must be computed.

The linear simplification also reduces the required training cost; since the bilinear
form scales linearly with a function of the parameter, we can train with a single
value of the parameter which allows to compute a decomposition of the bilinear form
matrix only once, then use it to solve the linear system for every snapshot, greatly
reducing the $O(N_h^2)$ cost of the $N_{samples}$ solves for port snapshots.

\section{Numerical Results}\label{sec:results}
We consider two lattice structure optimization problems; one small example,
for which performance of the component-wise model may be compared directly to the
solution of the conforming finite element model (the full-order model or FOM, which is
different from the CWFOM), and one optimization of a larger system for which the
solution of the FOM is infeasible. We use a fine discretization of components in order to
illustrate the capability of the component-wise methodology to capture a high level of detail in
simulations while preserving runtime that is asymptotically independent of the
underlying component discretization \cite{ballani2018component}.

\subsection{A cantilever beam with lattice structure}
The components used in this numerical example are pictured, with their
discretization, in Figure \ref{fig:two_components}. The second component is also
used in its vertical orientation; however, in order to make use the
simplification described in Section \ref{sec:linear_cwrom},  the
vertical orientation is in fact treated as a third component. This does not
increase the cost of the online model. The discretization of components consists
of first order quadrilateral elements with bilinear shape functions.  The
discretization of the joint component contains $3,675$ elements, while that of
the strut component contains $3,800$. Linear elasticity is approximated in two
dimensions using the plane stress approximation. The material used has a Young's
modulus of 69 GPa and a Poisson's ratio of 0.3, similar to the properties of
aluminum. All ports have a length of 1 cm, while the length of the strut component
is 5 cm.

\begin{figure}[h]
  \centering
  \subfloat[The joint component]
  {\includegraphics[width=0.25\textwidth]{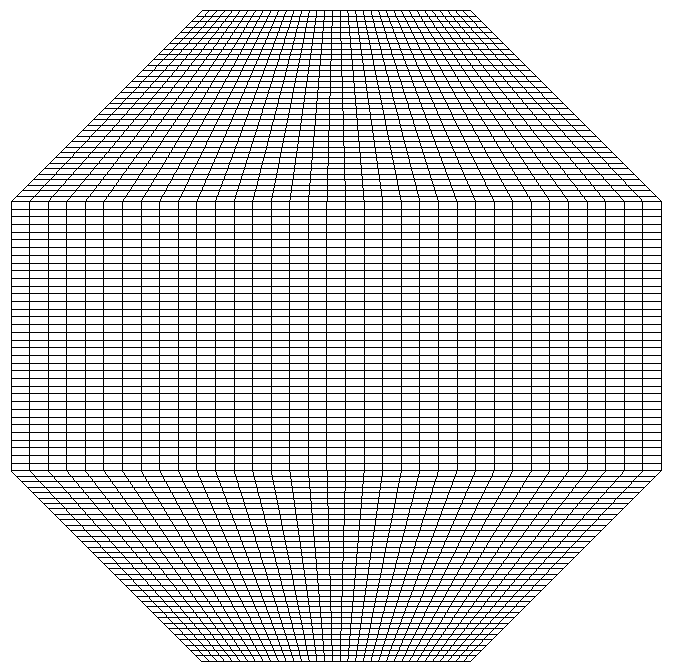}}\hspace{24pt}
  \subfloat[The strut component]
  {\includegraphics[width=0.5\textwidth]{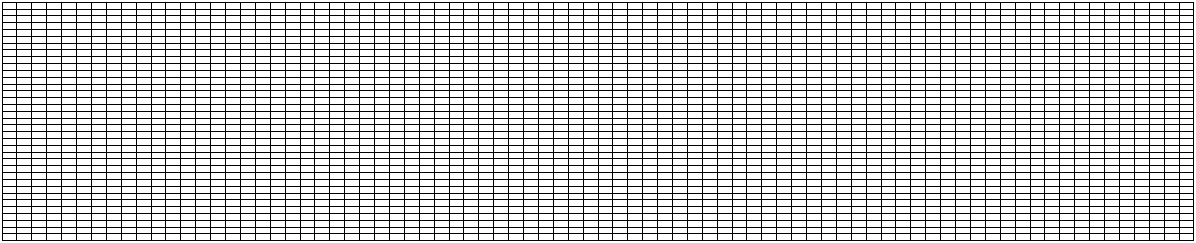}}
  \caption{Components used in the cantilever beam example}
  \label{fig:two_components}
\end{figure}

The lattice for this optimization example
contains a total of 290 components. The corresponding finite element model contains
\(1,844,640\) degrees of freedom.
We place the structure under tension by setting homogeneous Dirichlet boundary
conditions for displacement on the lower- and upper-most ports and the middle two ports on the left
hand side of the system and applying a uniform pressure force to each of the middle two
ports on the right-hand side. The upper port has a pressure force $100 \times 10^6$ N/m
in each of the positive X and Y directions; the lower port has a pressure force of $100
\times 10^6$ N/m applied in the positive X direction, and the same pressure force applied
in the negative Y direction. The length of each port is 1 cm, so the effective force on
each port is $\sqrt{2} \times 10^6$ N/m, directed at a $45^\circ$ angle upward for the
upper port and downward for the lower port. This problem setup is pictured in
\ref{fig:small_setup}.

The timings in this example are obtained on a desktop computer with an Intel
i7-4770k CPU on a single core operating at 3.5 GHz, and 32 GB of RAM, using an original
software implementation our methodology. The optimization method in
all examples is the method of moving asymptotes \cite{svanberg1987mma}, implemented in
the NLopt optimization library \cite{johnsonnlopt}.

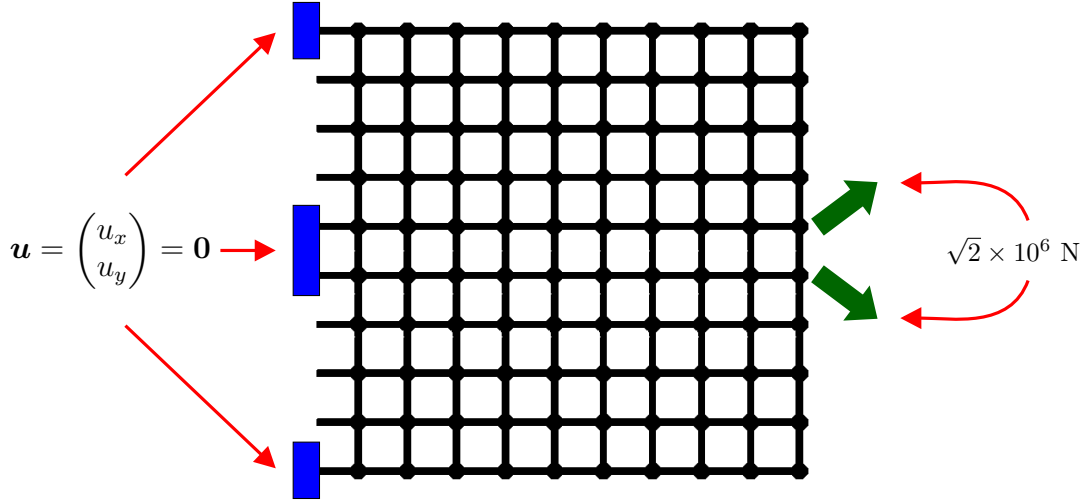
\begin{figure}[h]
  \centering
  \begin{tikzpicture}
    \node (syspic) {\includegraphics[width=0.4\textwidth]{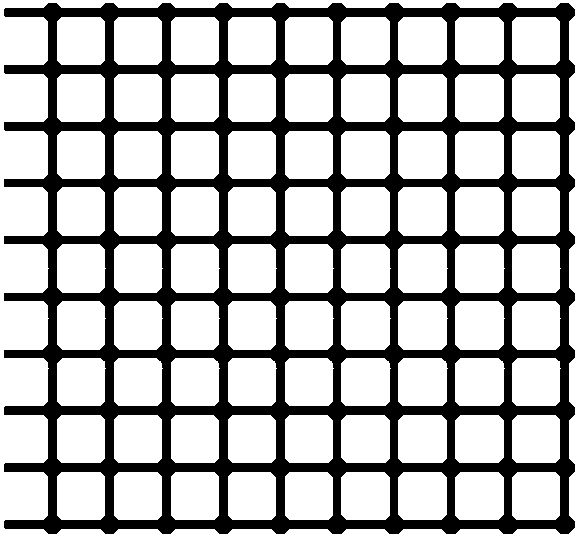}};

    \filldraw[fill=blue] (-3.57, 2.55) -- ++(0, 0.75) -- ++(0.35, 0) -- ++(0, -0.75) -- cycle;
    \filldraw[fill=blue] (-3.57, -2.55) -- ++(0, -0.75) -- ++(0.35, 0) -- ++(0, 0.75) -- cycle;
    \filldraw[fill=blue] (-3.57, -0.6) -- ++(0, 1.2) -- ++(0.35, 0) -- ++(0, -1.2) -- cycle;
    \node (homogeneous) at (-6, 0) {\large\(\bm{u} = \begin{pmatrix} u_x \\ u_y \end{pmatrix} = \bm{0}\)};
    \draw[-{Triangle[width=8pt, length=8pt]}, red, very thick] (-5.8, -1) -- (-3.8, -2.9);
    \draw[-{Triangle[width=8pt, length=8pt]}, red, very thick] (-5.8, 1) -- (-3.8, 2.9);
    \draw[-{Triangle[width=8pt, length=8pt]}, red, very thick] (homogeneous) -- (-3.8, 0.0);

    \draw[-{Triangle[width=16pt, length=10pt]}, line width=8pt, draw=black!60!green] (3.4, 0.31) -- (4.2, 0.91);
    \draw[-{Triangle[width=16pt, length=10pt]}, line width=8pt, draw=black!60!green] (3.4, -0.31) -- (4.2, -0.91);
    \node (force1) at (6, 0) {\(\sqrt{2} \times 10^6\) N};
    \draw[-{Triangle[width=8pt, length=8pt]}, red, very thick] (6.2, 0.4) to[out=110,in=0] (4.5, 0.9);
    \draw[-{Triangle[width=8pt, length=8pt]}, red, very thick] (6.2, -0.4) to[out=250,in=0] (4.5, -0.9);
  \end{tikzpicture}
  \caption{Problem setup for the cantilever beam numerical example}
  \label{fig:small_setup}
\end{figure}

\subsubsection{Performance of the component-wise discretization}
\label{sec:cwrom_perf}
In order to illustrate the superior performance of the component-wise modeling methodology,
we compare the solution time for the CWFOM and the CWROM to that required to solve the
underlying FEM problem on the same system discretization (denoted the FOM, by contrast with
the CWFOM, which is also a full-order model). This study is performed for the
problem shown in Figure \ref{fig:small_setup}. The linear solver used in all cases is a
sparse direct solver using the Cholesky factorization; future work will investigate the benefits
of using an alternative solver that better exploits the block sparse structure of the Schur
complement $\cwStiffness(\param)$.

In Figure \ref{fig:perf_comparison} we show comparison of the solution time and
solution error of CWROM's with varying port basis size relative to the FOM. The CWFOM is
also included for comparison;
the port basis dimension in the full order model is \(\spaceDimension{p}{\gamma} = 72\) for
all ports. Note that the relative solution and time and relative error are plotted on separate
ordinates; both are shown on a logarithmic scale. The relative solution time is 
given relative to the time for a FOM solve: $t_{rel} = t / t_{FOM}$ while the relative
error is in \(L^2\) norm; that is,
\begin{equation}\label{eq:L2_relative_error}
  \epsilon_{rel} = \frac{\left\|u_{CW} - u_{FOM}\right\|_{L^2}}{\left\|u_{FOM}\right\|_{L^2}}
\end{equation}
with the \(L^2\) norm given as usual by
\[
  \left\| u \right\|_{L^2} = \left(\int_\Omega u^T u \ d\Omega\right)^{1/2},
\]
where $u_{CW}$ is the solution as found using the component-wise model and $u_{FOM}$ is the
solution from the FOM.

\begin{figure}
  \centering
  \begin{tikzpicture}
    \pgfplotsset{set layers}
    \begin{axis}[
      scale only axis,
      xlabel={Port basis dimension},
      axis y line*=left,
      ylabel={$t_{rel}$ \ref{pgfplots:timeplot}},
      width=0.5\textwidth,
      ymode=log,
      log basis y = 10,
      xmin=2, xmax=22,
    ]
    \addplot
      coordinates {(4, 0.00084) (6, 0.0018) (8, .0023) (12, .004) (16, 0.0062) (20, .0097) }; 
    \label{pgfplots:timeplot}
    \end{axis}

    \begin{axis}[
      scale only axis,
      axis y line*=right,
      axis x line = none,
      ylabel={$\epsilon_{rel}$ \ref{pgfplots:accplot}},
      ymode=log,
      ymin=1e-9,
      ymax=1e-2,
      width=0.5\textwidth,
      xmin=2, xmax=22,
      ylabel near ticks
    ]
    \addplot+[sharp plot,red,mark=x,mark size=3]
      coordinates {(4, 5.7e-3) (6, 4.7e-3) (8, 2.8e-4) (12, 2.3e-5) (16, 8.7e-8) (20, 8e-9)}; 
    \label{pgfplots:accplot}
    \end{axis}
  \end{tikzpicture}
  \caption{Comparison of various CWROM resolutions' run time and error with respect to the FOM}
  \label{fig:perf_comparison}
\end{figure}
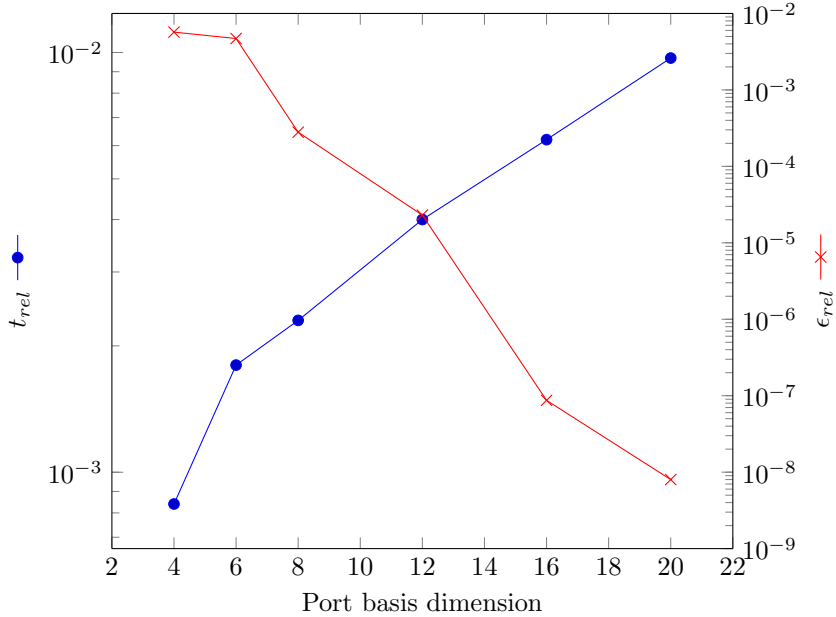

We note that the performance comparison given here should be taken as a general
statement of orders of magnitude that may be expected, as neither the finite
element primitives or the component-wise model implementation are well
optimized. The full data for Fig. \ref{fig:perf_comparison} may be found in
Table \ref{tab:perf_comparison}. The data points for port bases of sizes 36 and 72 are
omitted from the figure in order not to distort the abscissa; the trend in runtime
remains the same, while the relative error in the component-wise model does not
decrease appreciably when adding more than 20 basis vectors.

\begin{table}[h]
  \centering
  \begin{tabular}{| c | c  c  c |}
    \hline
    Port dimension & $t$ (s) & $t_{rel}$ & ${\epsilon_{rel}}$ \\
    \hline \hline
    4 & 0.081 & 8.4e-4 & 5.7e-3 \\
    6 & 0.17 & 1.8e-3 & 4.7e-3 \\
    8 & 0.22 & 2.3e-3 & 2.8e-4 \\
    12 & 0.38 & 4.0e-3 & 2.3e-5 \\
    16 & 0.59 & 6.2e-3 & 8.7e-8 \\
    20 & 0.93 & 9.7e-3 & 8.0e-9 \\
    36 & 3.2 & 3.3e-2 & 7.3e-9 \\
    72 & 16.32 & 1.7e-1 & 7.3e-9 \\
    \hline
  \end{tabular}
  \caption{Collected timing and error data}
  \label{tab:perf_comparison}
\end{table}

For this use case, with the simplification from Section \ref{sec:linear_cwrom}
in effect,
even the CWFOM achieves a 5x speedup over the FOM, and with the reduced basis approximation, a speedup of over 1000x
is realized while still achieving a relative error of less than 1\%. By a basis size
of 20, the relative error is reduced to $\approx 10^{-8}$, while still achieving
a 100x speedup. The small relative error present even in the CWFOM is introduced due
to finite precision arithmetic.

\subsubsection{Optimization results}
We solve the optimization formulation in Eq. \eqref{eq:opt_formulation} using
the method of moving asymptotes (MMA) \cite{svanberg1987mma}. We choose the
CWROM with 8 basis functions per port as a compromise between the time per
optimization iteration and error in the computed compliance and compliance
sensitivity. The maximum mass fraction for this optimization is taken to be
\(\frac{v_u}{v_t} = 60\%\), with \(v_t = \sum_i v_i\) the total volume of all
components. To check the sensitivity of the optimization to initial conditions,
the initial condition for the optimizer is sampled from a truncated
normal distribution with mean \(0.6\) and standard deviation \(0.05\); this
distribution is truncated so that values lie in \([0, 1]\) as required. This
initial value is used in order to approximately satisfy the volume constraint.
Several optimizer runs were used from different randomly sampled initial values,
and an initialization with the density for each component set to $\paramk{i} = 0.6$ to exactly
satisfy the constraint was also tested.

The final result from all initialization choices is similar, except for a
small fraction of the random initializations in which the optimization results in
a dramatic increase of the objective function then no subsequent improvement.
However, the MMA iteration does not succeed in finding a local minimum, but stalls,
making only very small adjustments to the state while searching for a descent direction.
Future work will investigate alternative optimization algorithms to overcome this
shortcoming. To handle this difficulty, we terminate the optimization based on a running
mean of the change in the parameter values. Specifically, the optimization is ended when
the mean value over the past 10 iterations of the scaled norm
\begin{equation}
  \label{eq:stopping_criterion}
  \frac{\left\|\param^n - \param^{n-1}\right\|_2}{\sqrt{n_I}}
\end{equation}
is less than $10^{-6}$, where $\param^n$ denotes the value of the optimization
parameter at the $n$th iteration, and $n_I$ is the number of instantiated components
(dimension of the parameter vector). In most cases, this criterion was met at between
50 and 150 MMA iterations.

\begin{figure}[h]
  \centering
    \subfloat[The density field after optimization]
    {\includegraphics[width=0.53\textwidth]{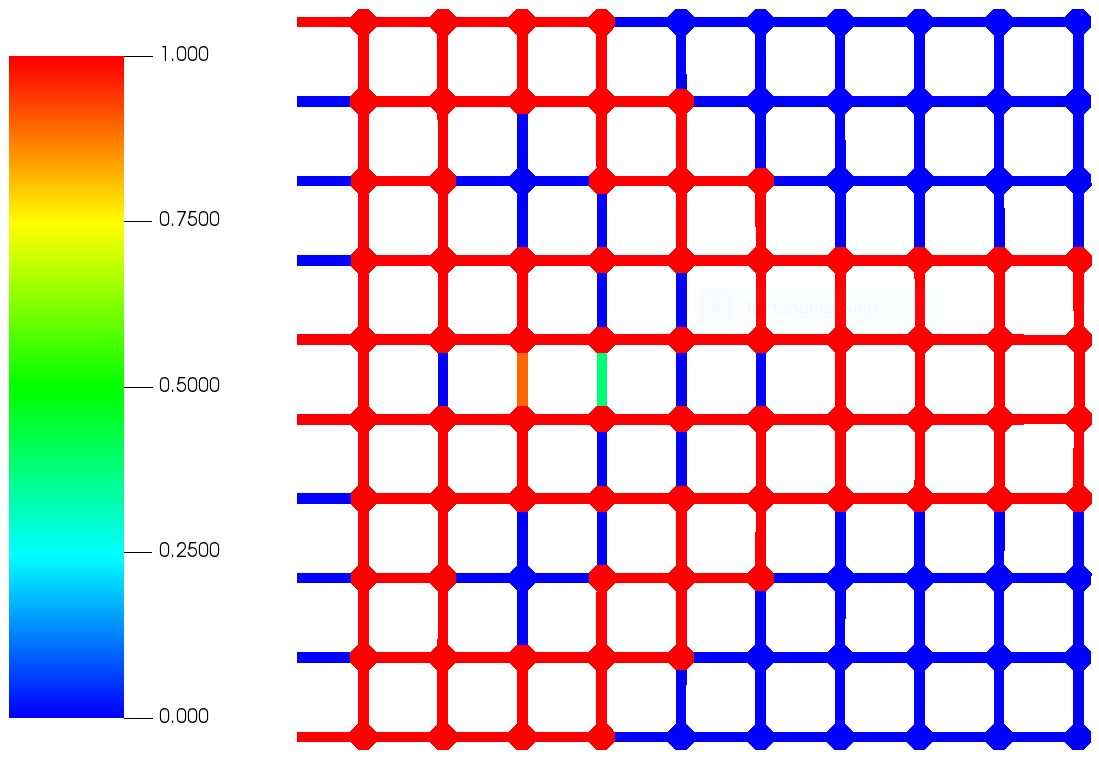}\hspace{40pt}}
    \subfloat[The system design obtained from post-processing]
    {\includegraphics[width=0.38\textwidth]{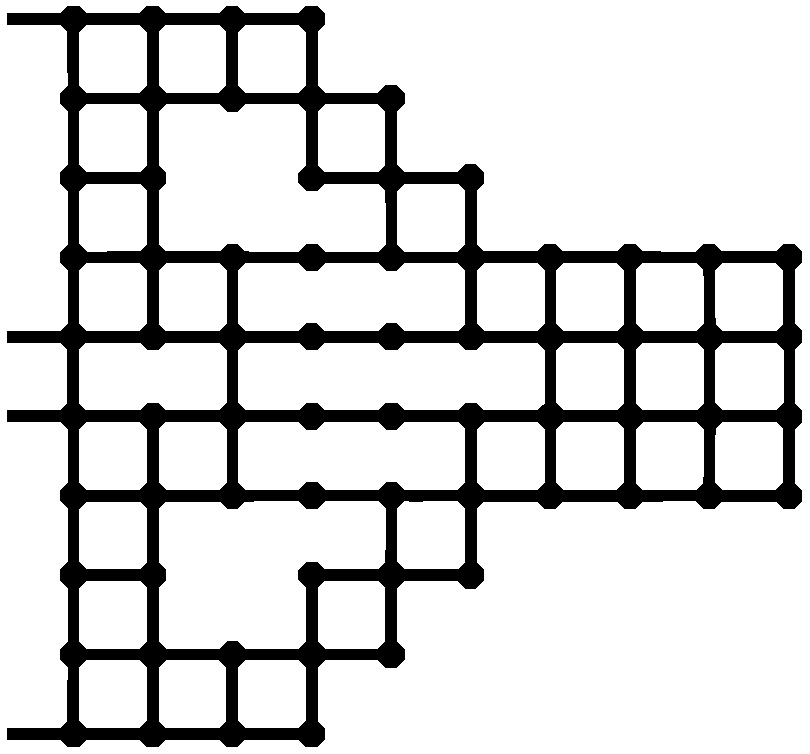}}
    \caption{Optimization result from a uniform initial condition}
    \label{fig:small_optimization_result}
\end{figure}

The SIMP penalization does not result in a purely black-and-white design;
rather, there are a few components with intermediate values of \(\paramk{i}\).
Therefore, post-processing is required to create a design that consists only of
solid material and void regions. We choose to remove those components with
\(\paramk{i} < 0.7\), and set \(\paramk{i} = 1\) for the remaining components.
The numerical results reported here are for a uniform initial condition for the
optimizer. The initial value of compliance was \(9,858\) N\(\cdot\)m, the
optimized value was \(2,185\) N\(\cdot\)m, and following post-processing the
system by removing partially-void components, the obtained compliance value was
still \(2,185\) N\(\cdot\)m to the fourth significant digit, and the result
satisfies the \(60\%\) volume fraction constraint.  The configuration of the
optimized system is shown in Fig. \ref{fig:small_optimization_result}, with both
the density field returned by the optimizer and the resulting structure.  Random
initialization does affect the final optimization result. For most choices of
the random initialization (in which the optimization converges), there are
several asymmetries in the parameter field; when beginning with a uniform
initialization, the optimized structure is symmetrical. There is more than a
single asymmetrical local optimum, as well. A few realizations of the random
initial condition result in a final design that slightly improves on the
compliance of the symmetrical design; however, the improvement is not retained
after post-processing, because of components with intermediate densities that
are removed. Instead, the final result is worse; therefore, we recommend using a
uniform initial condition. Of course, this observation is mainly for MMA and
for other optimization solvers, we may see different results.

\subsection{A cantilever beam with a larger number of lattice components}
The larger optimization example makes use of the same component discretizations and
material properties as previously, but a lattice containing 2,950 components. This
lattice is pictured in Fig. \ref{fig:large_lattice}. The finite element mesh of the full
system contains 9,362,520 nodes, for a total of 18,725,040 degrees of freedom in the
system without static condensation. In this example, we compare the optimized solutions
obtained from increasing resolution in the reduced order model; as in the comparison of
performance for the forward model, we use port basis sizes of 4, 6, 8, 12, 16, and 20
basis functions. The total number of degrees of freedom in the CWROM then ranges from
16,080 to 80,400.

The system of components in this study form a cantilever beam with a homogeneous Dirichlet
boundary condition on each of the ports on the left hand side of the domain.
Two loads are placed on the system, on the right and bottom
ports of the lowest right component. On the right port, a uniform pressure force of $10^7$ N/m to
the right is added, while on the bottom port we impose a uniform pressure force of
$3 \times 10^7$ N/m downward.

\begin{figure}[h]
  \centering
  \includegraphics[width=0.9\textwidth]{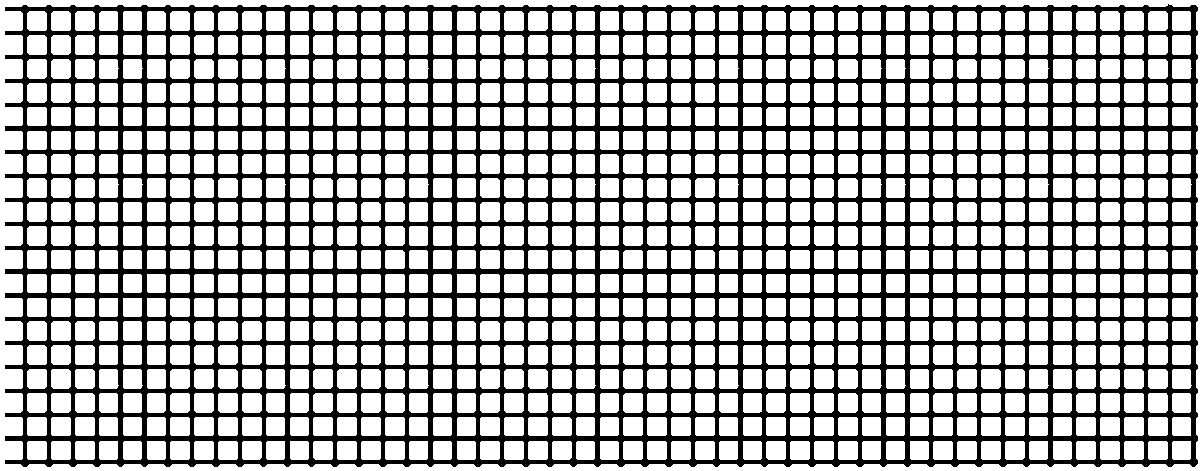}
  \caption{The lattice structure for the cantilever beam problem}
  \label{fig:large_lattice}
\end{figure}

The optimization is solved using the same optimization method and stopping criterion as the previous example;
however, in this case we do not use a random initialization for multiple runs, but instead compare the solutions
produced from a uniform initialization by different CWROM resolutions.
In addition, we increase the stopping tolerance for change in the optimization parameter (Eq. \eqref{eq:stopping_criterion})
to $10^{-4}$. The target mass fraction for this optimization was 25\%. When post-processing, the tolerance for intermediate
densities was increased, by removing only components
with $\paramk{i} < 0.5$ and setting $\paramk{i} = 1$ for the rest; this results in a post-processed
design that slightly exceeds the given upper bound for the mass fraction. Table \ref{tab:opt_comparison} shows
a comparison of the results obtained from different CWROM resolutions; the first column for compliance results is
the value of the optimized compliance computed by the CWROM without post-processing, while the value in the final
column is the value of the compliance computed by the highest-resolution CWROM using the post-processed densities.
We also report the relative error for each discretization, measured against the solution of the CWFOM; the relative
error is again given by Eq.~\eqref{eq:L2_relative_error}. We note that the magnitude of the relative error is
consistent with the results in the previous section, up until the basis size of 20; at very high resolutions of
the CWROM, its solution matches the CWFOM more closely than it matched the full-order FEM model in the previous
section. Some numerical error is still present.  We only report optimization results up to a basis size
of 20 here both due to the runtime of the optimization in our experimental implementation, and
because, as shown below, the solution of the CWROM by a basis size of 20 is essentially
the same as that of the CWFOM.

\begin{table}[h]
  \centering
  \begin{tabular}{| c | c  c  c  c |}
    \hline
    Port dimension & Time (s) & $c_{ROM}$ (N$\cdot$m) & $c_{post}$ (N$\cdot$m) & Relative error\\
    \hline \hline
    4 & 38 & 8980.3 & 9051.5 & $1.04\times 10^{-2}$ \\
    6 & 36 & 8989.5 & 8988.1 & $7.83\times 10^{-3}$ \\
    8 & 101 & 8947.6 & 8882.5 & $2.88\times 10^{-4}$ \\
    12 & 396 & 8951.6 & 8880.3 & $2.43\times 10^{-5}$ \\
    16 & 558 & 9066.4 & 8965.5 & $1.32\times 10^{-7}$ \\
    20 & 924 & 9066.3 & 8965.5 & $3.81\times 10^{-10}$ \\
    36 & -- & -- & -- & $2.20\times 10^{-10}$ \\
    \hline
  \end{tabular}
  \caption{Optimization results for the cantilever beam problem. The time given is
  the total time used for the optimization procedure. $c_{ROM}$ is the optimal compliance
  returned from the optimizer. $c_{post}$ is the compliance of the
  optimized design computed after post-processing, using the CWROM with port basis
  dimension 20. Error reported is relative to a reference solution of the CWFOM.}
  \label{tab:opt_comparison}
\end{table}

A port basis size of 6 was faster than a port basis size of 4 due to faster convergence of the
optimization; this trend did not continue for larger basis sizes. Up to a basis size of 12, the
compliance of the optimized design after post-processing decreased, however, the two higher-dimensional
CWROM's actually resulted in a substantially worse design. This irregularity may indicate that the
lower-dimensional CWROM smooths the objective function, allowing the the optimizer to find a lower
local minimum than for higher basis sizes. There is a trade-off between time to solution and accuracy
evident here; a basis size of 4 is probably not accurate enough for most purposes, as evidenced by the
discrepancy between the compliance value returned by the optimizer and the significantly higher value
resulting after post-processing. For a basis size of 6 or greater, however, the largest difference in
compliance of the optimized designs is only approximately 1\%, while the computational cost of the
optimization increases rapidly.

We note that while the complexity estimates in Section
\ref{sec:computational-costs} indicate that the runtime should increase quadratically with basis
size, we actually observe a more rapid increase in practice, making the acceleration enabled by the
CWROM even more significant.

In Fig. \ref{fig:large_densities}
and Fig. \ref{fig:large_structure} we show the resulting density field and post-processed optimized
structure obtained from the basis size 12 optimization (which resulted in the best design).
The optimization was terminated based on the stopping criterion given in Eq. \ref{eq:stopping_criterion}
after 153 iterations. The value of the compliance after optimization
was 8,952 N$\cdot$m, and post-processing did not remove any components with non-zero densities less than $0.5$,
and set the densities of 74 components with intermediate values of $\paramk{i}$ to 1. This resulted in a
slightly lower compliance value in the post-processed structure, 8,880 N$\cdot$m, at the cost of slightly
exceeding the target mass fraction; the mass fraction of the post-processed structure is 25.3\%.

\begin{figure}[H]
  \centering
  \subfloat[Volume fraction legend]
  {\includegraphics[width=0.09\textwidth]{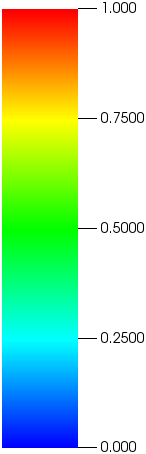}}\hspace{24pt}
  \subfloat[Volume fraction field for the lattice structure after optimization]
  {\includegraphics[width=0.75\textwidth]{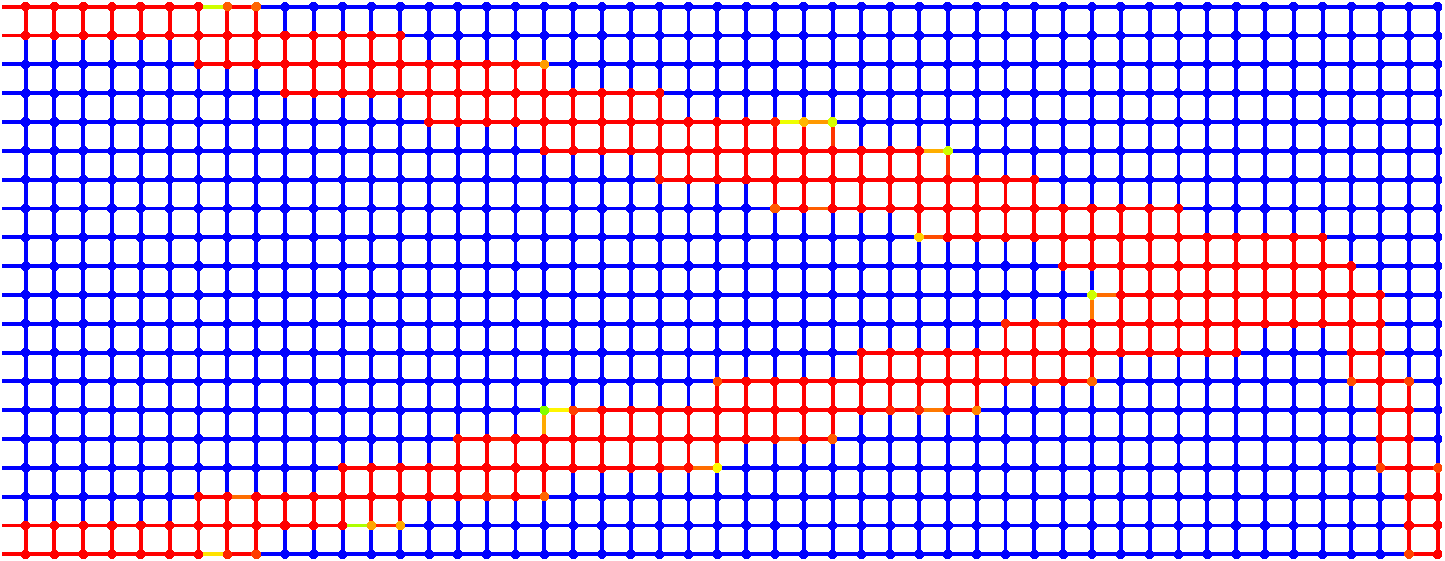}}
  \caption{}
  \label{fig:large_densities}
\end{figure}

\begin{figure}[h]
  \centering
  \includegraphics[width=0.8\textwidth]{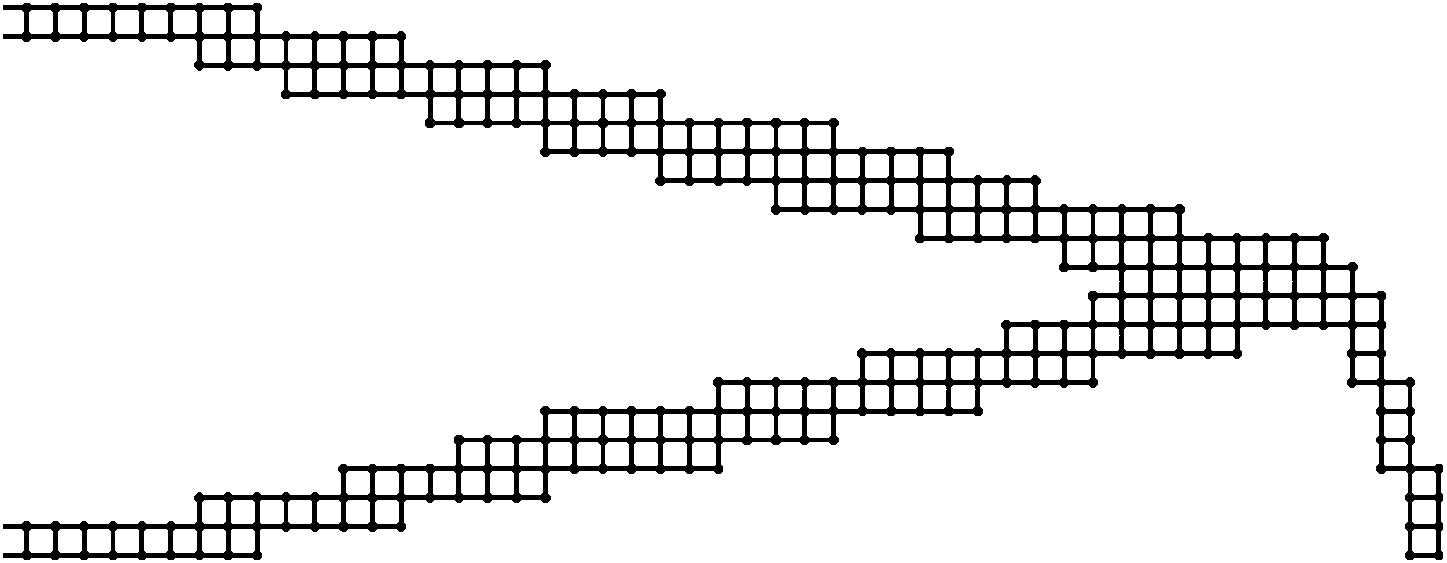}
  \caption{The optimized structure after post-processing}
  \label{fig:large_structure}
\end{figure}

In Fig. \ref{fig:stress_field}, we show the Von Mises stress field over a single component from the
optimized system in \ref{fig:large_structure}. The resolved stress concentrations at the corners that
are apparent in this plot illustrate a key benefit of the component-wise modeling approach for TO -
it does not sacrifice high resolution solutions for the sake of speedup. In contrast with homogenization
methods, which take the periodic structure to be at less than the length scale of a finite element, our
approach allows computation of fields at the same length scale as the lattice structure. Approaches
based on approximating lattice members as beam elements, on the other hand, solve an approximated form
of the governing equation, while the CWROM solves the original problem in a reduced function space that
nevertheless closely approximates the full FEM space as seen from the relative accuracies reported in
Section \ref{sec:cwrom_perf}.

\begin{figure}[h]
  \centering
  \subfloat[Von Mises stress scale, in N/$\text{m}^2$]
  {\includegraphics[width=0.16\textwidth]{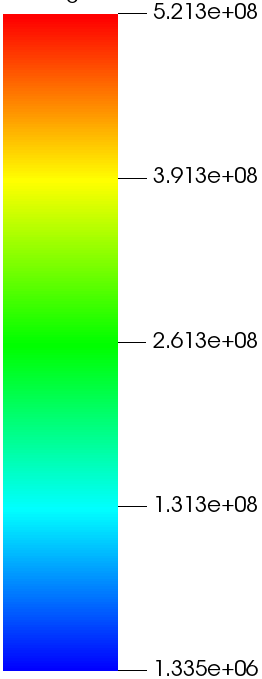}}\hspace{24pt}
  \subfloat[The von Mises stress field on the joint component neighboring the
  upper left strut in the domain]
  {\includegraphics[width=0.45\textwidth]{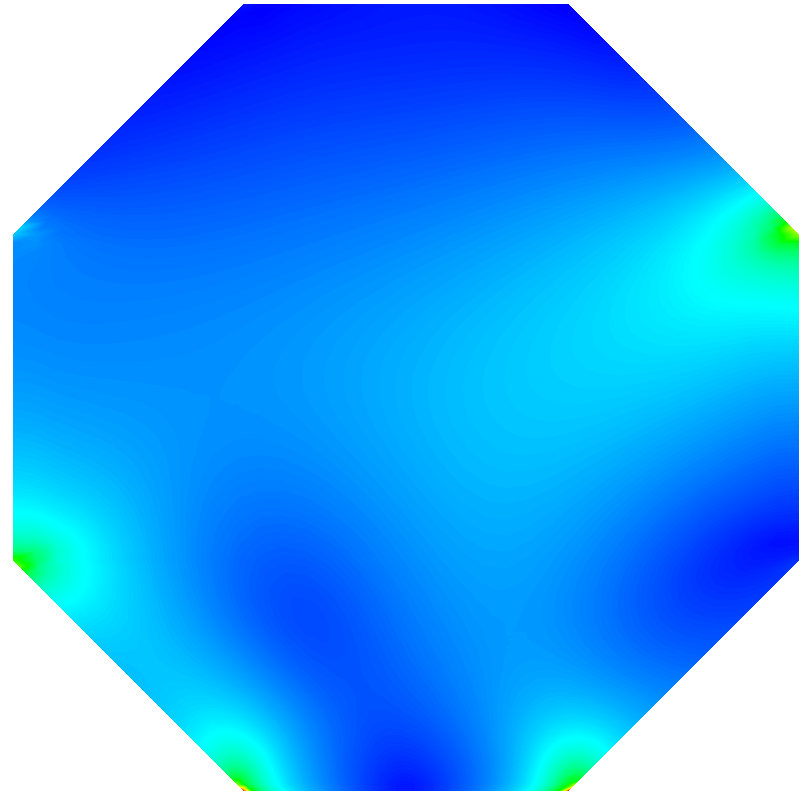}}
  \caption{}
  \label{fig:stress_field}
\end{figure}

\section{Conclusion}\label{sec:conclusion}
We have demonstrated a component-wise topology optimization method that provides
a combination of computational efficiency and accuracy not seen in other methods
for lattice structure optimization.  Using the static condensation formulation
of Huynh et al. \cite{huynh2013static} and the port reduction of Eftang \&
Patera \cite{eftang2013pr}, we obtain a speedup in the solution of the forward
model of over 1000x over a conforming FEM model, with relative error of less
than 1\%.  We also show a key simplification of the component-wise method for
the case where the parameter dependence of the linear elasticity weak form is
linear in a function of parameter, providing further acceleration in our
lattice structure design optimization using a SIMP parameterization. This work
forms a base on which to build more sophisticated component-wise optimization
frameworks; the high-resolution capability of the component-wise reduced order
model is particularly intriguing in the context of stress-based TO. Future work
will develop a stress-based component-wise formulation, and extend the SIMP
parameterization used in this paper to include geometric parameters modifying
the shape of members in the lattice structure.

\section*{Acknowledgments}
This work was performed in part at Lawrence Livermore National Laboratory and
was supported by the LDRD program fundings (i.e., 17-ER-026 and 20-FS-007).
Lawrence Livermore National Laboratory is operated by Lawrence Livermore
National Security, LLC, for the U.S. Department of Energy, National Nuclear
Security Administration under Contract DE-AC52-07NA27344 and LLNL-JRNL-815816.
This work was supported in part by the AEOLUS center under US Department of
Energy Applied Mathematics MMICC award DE-SC0019303.

  \bibliographystyle{plain}
  \bibliography{references}

\end{document}

%% file: commands.tex
\usepackage{tikz}

\newcommand{\keywords}[1]{\small
  \textbf{\textit{Keywords---}} #1
}
\usepackage{enumitem}

\newlist{Assumption}{enumerate}{1}
\setlist[Assumption]{label=A\arabic*}

\newtheorem{remark}{Remark}
\newtheorem{thm}{Theorem}

\DeclareMathOperator*{\minimize}{minimize}

\newcommand{\bmat}[1]{\begin{bmatrix}#1\end{bmatrix}} 
\newcommand{\pmat}[1]{\begin{pmatrix}#1\end{pmatrix}} 
\definecolor{Blue}{rgb}{0,0,1}
\definecolor{Red}{rgb}{1,0,0}
\definecolor{Green}{rgb}{0,1,0}

\newcommand{\canonicalk}[1]{\boldsymbol{e}_{#1}}

\newcommand{\NN}{\mathbb{N}}
\newcommand{\RR}[1]{\ensuremath{\mathbb{R}^{ #1 }}}
\newcommand{\HH}{H^1}
\newcommand{\mapto}{\rightarrow}

\newcommand{\Span}[1]{\mathrm{span}\{#1\}}

\newcommand{\nat}[1]{\NN({#1})}
 
\newcommand{\paramSymbol}{\mu}
\newcommand{\paramSymbolk}[1]{\paramSymbol_{#1}}
\newcommand{\param}{\boldsymbol{\paramSymbol}}
\newcommand{\paramSymbolLow}{\paramSymbol_{\ell}}
\newcommand{\paramk}[1]{\param_{#1}}
\newcommand{\paramDomain}{\mathcal{D}}
\newcommand{\paramDomaink}[1]{\paramDomain_{#1}}
\newcommand{\sysDomain}{\Omega}
\newcommand{\sysDomaink}[1]{\sysDomain_{#1}}
\newcommand{\sysDomainClosure}{\bar{\sysDomain}}
\newcommand{\sysDomainClosurek}[1]{\sysDomainClosure_{#1}}
\newcommand{\refDomain}{\hat{\sysDomain}}
\newcommand{\refDomaink}[1]{\refDomain_{#1}}
\newcommand{\solSymbol}{u}
\newcommand{\rsol}{\tilde{\solSymbol}_h}
\newcommand{\sol}{\solSymbol_h}
\newcommand{\solSpaceSymbol}{X}
\newcommand{\solSpace}{\solSpaceSymbol}
\newcommand{\femSpace}{\solSpaceSymbol^h}
\newcommand{\cwSol}{\mathtt{U}}
\newcommand{\reducedcwSol}{\reduced{\cwSol}}
\newcommand{\cwSolk}[1]{\cwSol_{#1}}

\newcommand{\bub}[2]{b_{#1}^{#2}}
\newcommand{\dfbubk}[1]{\boldsymbol{b}_{#1}^f}
\newcommand{\bubres}[1]{\boldsymbol{R}_{#1}}
\newcommand{\bubspace}[2]{B_{#1}^{#2}}
\newcommand{\ifunc}[1]{\Phi^h_{#1}}
\newcommand{\rifunc}[1]{\tilde{\Phi}^h_{#1}}
\newcommand{\femSpacek}[1]{\femSpace_{#1}}
\newcommand{\boundary}[1]{\partial{#1}}
\newcommand{\activeSymbol}{A}
\newcommand{\inactiveSymbol}{I}

\newcommand{\reduced}[1]{\tilde{#1}}
\newcommand{\extended}[1]{\bar{#1}}
\newcommand{\portSpace}{P^h}
\newcommand{\portSpacek}[1]{\portSpace_{#1}}
\newcommand{\spaceDimension}[2]{\mathcal{N}_{#1}^{#2}}
\newcommand{\reducedSpaceDimension}[2]{\reduced{\mathcal{N}}_{#1}^{#2}}
\newcommand{\portSubspace}{\hat{P}^h}
\newcommand{\portSubspacek}[1]{\portSubspace_{#1}}

\newcommand{\compIndexMap}{\mathcal{R}}
\newcommand{\compIndexMapk}[1]{\mathcal{R}\left(#1\right)}
\newcommand{\transMap}{\mathcal{I}}
\newcommand{\transMapk}[1]{\transMap_{#1}}
\newcommand{\globpmap}{\mathcal{G}}
\newcommand{\globpmapk}[1]{\globpmap_{#1}}

\newcommand{\bilinearSymbol}{a}
\newcommand{\sbilinearSymbol}{\bar{a}}
\newcommand{\bilinear}[1]{\bilinearSymbol\left({#1}\right)}
\newcommand{\bilineark}[2]{\bilinearSymbol_{#1}\left({#2}\right)}

\newcommand{\sbilineark}[2]{\sbilinearSymbol_{#1}\left({#2}\right)}
\newcommand{\linearSymbol}{f}
\newcommand{\linear}[1]{\linearSymbol\left({#1}\right)}
\newcommand{\lineark}[2]{\linearSymbol_{#1}\left({#2}\right)}
\newcommand{\testfun}{v}
\newcommand{\bm}[1]{\boldsymbol{#1}}

\newcommand{\femForce}{\boldsymbol{F}}
\newcommand{\femForcek}[1]{\femForce_{#1}}
\newcommand{\femStiffness}{\boldsymbol{K}}
\newcommand{\femStiffnessk}[1]{\femStiffness_{#1}}
\newcommand{\sfemStiffnessk}[1]{\bar{\femStiffness}_{#1}}

\newcommand{\instaindex}{i}
\newcommand{\refindex}{r}
\newcommand{\portindex}{j}
\newcommand{\numInsta}{n_I}
\newcommand{\numRef}{n_R}
\newcommand{\numSamples}{N_{\text{samples}}}

\newcommand{\portSymbol}{\gamma}
\newcommand{\rport}{\hat{\portSymbol}}
\newcommand{\rportk}[1]{\rport_{#1}}
\newcommand{\portbasis}[1]{\chi_{#1}}

\newcommand{\gportSymbol}{\gamma}

\newcommand{\zerovec}{\boldsymbol{0}}
\newcommand{\lagrangeMultiplier}{\boldsymbol{\lambda}}

\newcommand{\bubLagrangeMultiplier}[1]{\boldsymbol{\beta}_{#1}}

\newcommand{\reducedSolution}{\reduced{u}_h}
\newcommand{\reducedCWsol}{\reduced{\cwSol}}
\newcommand{\extendedCWsol}{\extended{\cwSol}}
\newcommand{\reducedCWsolk}[1]{\reduced{\cwSol}_{#1}}
\newcommand{\extendedCWsolk}[1]{\extended{\cwSol}_{#1}}
\newcommand{\CWsolk}[1]{\cwSol_{#1}}
\newcommand{\interfacek}[1]{\Phi^h_{#1}}

\newcommand{\portReducedSet}{\mathcal{Q}}

\newcommand{\complianceSymbol}{c}
\newcommand{\compliance}{\complianceSymbol}
\newcommand{\compliancek}[1]{\compliance_{#1}}
\newcommand{\dcomplianceSymbol}{\bar{\complianceSymbol}}
\newcommand{\dcompliance}{\dcomplianceSymbol}
\newcommand{\dcompliancek}[1]{\dcompliance_{#1}}
\newcommand{\acomplianceSymbol}{\tilde{\complianceSymbol}}
\newcommand{\acompliance}{\acomplianceSymbol}

\newcommand{\dacomplianceSymbol}{\hat{\complianceSymbol}}
\newcommand{\dacompliance}{\dacomplianceSymbol}

\newcommand{\rcompliance}{\tilde{\complianceSymbol}}

\newcommand{\legendre}[1]{L_{#1}^k}
\newcommand{\errSymbol}{e}
\newcommand{\cwerrSymbol}{\mathtt{e}}
\newcommand{\complianceErrSymbol}{\errSymbol_{\complianceSymbol}}
\newcommand{\complianceSensitivityErrSymbol}{\errSymbol_{\complianceSymbol s}}

\newcommand{\solerr}{\errSymbol^h}
\newcommand{\cwsolerr}{\cwerrSymbol^h}
\newcommand{\compErr}{{\complianceErrSymbol}^h}
\newcommand{\compSensErr}{{\complianceSensitivityErrSymbol}^h}

\newcommand{\skeletonSpace}{\mathcal{S}}
\newcommand{\elasticityTensor}{\mathbb{C}}

\newcommand{\cwStiffness}{\mathtt{K}}
\newcommand{\reducedcwStiffness}{\reduced{\cwStiffness}}

\newcommand{\scwStiffness}{\bar{\cwStiffness}}
\newcommand{\cwForce}{\mathtt{F}}
\newcommand{\cwRes}{\mathtt{R}}

\newcommand{\energyNorm}[1]{\| {#1} \|_{\param}}
\newcommand{\KinducedNorm}[1]{\| {#1} \|_{\cwStiffness}}
\newcommand{\ltwoNorm}[1]{\| {#1} \|_2}
\newcommand{\absNorm}[1]{| {#1} |}
\newcommand{\dummyVec}{\boldsymbol{v}}
\newcommand{\dummyFunc}{v}
\newcommand{\stiffnessSensitivityNorm}{\kappa}
\newcommand{\normreducedCWsol}{\nu}

\newcommand{\drawhstrutcomp}[3]{
    \begin{scope}[shift={({(#1)*(#2)}, {(#1)*(#3)})},scale=#1]
        \draw (0, 0) -- (5, 0) -- (5, 1) -- (0, 1) -- cycle;
    \end{scope}
}

\newcommand{\drawvstrutcomp}[3]{
    \begin{scope}[shift={({(#1)*(#2)}, {(#1)*(#3)})},scale=#1]
        \draw (0, 0) -- (1, 0) -- (1, 5) -- (0, 5) -- cycle;
    \end{scope}
}

\newcommand{\drawjointcomp}[3]{
    \begin{scope}[shift={({(#1)*(#2)}, {(#1)*(#3)})},scale=#1]
        \draw (0, 0) -- ++(1, 0) -- ++({1/sqrt(2)},{1/sqrt(2)}) -- ++(0, 1) -- ++({-1/sqrt(2)},{1/sqrt(2)}) -- ++(-1, 0) -- ++({-1/sqrt(2)},{-1/sqrt(2)}) -- ++(0, -1) -- cycle;
    \end{scope}
}

\newcommand{\rowshift}{(6+sqrt(2))}

\newcommand{\smallsystem}[3]{
    \drawjointcomp{#1}{#2}{#3}
    \drawhstrutcomp{#1}{#2 + 1 + 1/sqrt(2)}{#3 + 1/sqrt(2)}
    \drawjointcomp{#1}{#2 + \rowshift}{#3}
    \drawvstrutcomp{#1}{#2}{#3 + 1 + sqrt(2)}
    \drawjointcomp{#1}{#2}{#3 + \rowshift}
    \drawhstrutcomp{#1}{#2 + 1 + 1/sqrt(2)}{#3 + 1/sqrt(2) + \rowshift}
    \drawjointcomp{#1}{#2 + \rowshift}{#3 + \rowshift}
    \drawvstrutcomp{#1}{#2 + \rowshift}{#3 + 1 + sqrt(2)}
}